\documentclass[a4paper,12pt,oneside,dvipsnames,reqno]{amsart} 

\usepackage{amsmath,amsfonts,amsthm,amssymb}

\usepackage{xcolor}
\usepackage{graphicx}
\usepackage{tensor}							
\usepackage{cite}

\usepackage{bbm}

\usepackage{hyperref}
\hypersetup{colorlinks=true, citecolor=PineGreen, linkcolor=RoyalBlue, linktoc=page,urlcolor=RoyalBlue}

\usepackage{geometry}

\theoremstyle{plain}
\newtheorem{theorem}{Theorem}[section] 
\newtheorem{lem}[theorem]{Lemma}
\newtheorem{prop}[theorem]{Proposition}
\newtheorem{cor}[theorem]{Corollary}
\newtheorem{ass}{Assumption}[section]

\theoremstyle{remark}
\newtheorem{rem}[theorem]{Remark}

\theoremstyle{definition}
\newtheorem{defn}{Definition}[section]

\newcommand{\A}{\mathbb{A}}
\newcommand{\Ax}{\mathbb{A}^{\times}}

\newcommand{\R}{\mathbb{R}}
\newcommand{\Rx}{\mathbb{R}^{\times}}
\newcommand{\C}{\mathbb{C}}
\newcommand{\Cx}{\mathbb{C}^{\times}}

\newcommand{\Z}{\mathbb{Z}}
\newcommand{\Zp}{\mathbb{Z}_{p}}
\newcommand{\Zpx}{\mathbb{Z}_{p}^{\times}}
\newcommand{\Q}{\mathbb{Q}}
\newcommand{\Qx}{\mathbb{Q}^{\times}}
\newcommand{\Qp}{\mathbb{Q}_{p}}
\newcommand{\Qpx}{\mathbb{Q}_{p}^{\times}}

\newcommand{\Bc}{\mathcal{B}}
\newcommand{\Cc}{\mathcal{C}}

\newcommand{\Ccinf}{\mathcal{C}^{\infty}_{c}}

\newcommand{\Wh}{\mathcal{W}}

\newcommand{\me}{\mathfrak{m}}
\newcommand{\Me}{\mathcal{M}}
\newcommand{\V}{\mathcal{V}}

\newcommand{\Xf}{\mathfrak{X}}

\newcommand{\of}{\mathfrak{o}}
\newcommand{\ofx}{\mathfrak{o}^{\times}}

\newcommand{\pf}{\mathfrak{p}}

\newcommand{\Hf}{\mathfrak{H}}

\newcommand{\Ga}{\mathfrak{G}}

\newcommand{\Fx}{F^{\times}}
\newcommand{\Fxh}{\hat{F}^{\times}}

\newcommand{\GL}{\operatorname{GL}}

\newcommand{\SL}{\operatorname{SL}}

\renewcommand{\Re}{\operatorname{Re}}
\renewcommand{\Im}{\operatorname{Im}}

\newcommand{\Ind}{\operatorname{Ind}}

\newcommand{\sgn}{\operatorname{sgn}}
\newcommand{\isom}{\cong}

\newcommand{\Stab}{\operatorname{Stab}}

\newcommand{\Vol}{\operatorname{Vol}}

\newcommand{\supp}{\operatorname{supp}}

\newcommand{\diag}{\operatorname{diag}}

\newcommand{\Id}{\operatorname{Id}}

\newcommand{\IT}{\operatorname{IT}}

\newcommand{\Char}{\operatorname{Char}}
\newcommand{\lcm}{\operatorname{lcm}}

\newcommand{\bs}{\backslash}

\newcommand{\Mod}[1]{\ (\operatorname{mod}\ #1)}

\newcommand{\dxy}{d^{\times}y}
\newcommand{\dxx}{d^{\times}x}

\newcommand{\Kl}{\mathcal{K}\ell}
\newcommand{\KL}{\operatorname{KL}}

\newcommand{\mat}[4]{\left(\begin{array}{cc}
#1 & #2 \\ #3 & #4
\end{array}\right)}

\newcommand{\smat}[4]{\left(\begin{smallmatrix}
#1 & #2 \\ #3 & #4
\end{smallmatrix}\right)}

\newcommand{\tran}[2]{\tensor[^t]{#1}{^{#2}}}

\newcommand{\abs}[1]{\lvert{#1}\rvert}

\newcommand{\ifs}{\text{if }}
\newcommand{\andy}{\quad\text{and}\quad}

\title[Vorono\u{\i} summation for $\GL_{n}$]{Vorono\u{\i} summation for $\GL_{n}$: collusion between level and modulus}

\author{Andrew Corbett}

\newcommand{\Date}{1$^{\mathrm{st}}$ July 2018}
\date{\Date}

\address{Mathematisches Institut, Bunsenstr.\ 3-5, 37073 G{\"o}ttingen}
\email{andrew.corbett@uni-goettingen.de}

\begin{document}

\begin{abstract}

We investigate the Vorono\u{\i} summation problem for $\GL_{n}$ in the level aspect for $n\geq 2$. Of particular interest are those primes at which the level and modulus are jointly ramified -- a common occurrence in analytic number theory when using techniques such as the Petersson trace formula. Building on previous legacies, our formula stands as the most general of its kind; in particular we extend the results of Ichino--Templier \cite{ichino-templier}. We also give (classical) refinements of our formula and study the $p$-adic generalisations of the Hankel transform.


\end{abstract}

\maketitle

\setcounter{tocdepth}{1}
\tableofcontents

\section{Developments in the \texorpdfstring{Vorono\u{\i}}{Voronoi} summation problem for \texorpdfstring{$\GL_{n}$}{GL(n)}}\label{sec:introduction}

The Vorono\u{\i} summation formula for $\GL_{n}$ with $n\geq 2$ is a relative of the Poisson summation formula, which describes the $n=1$ case. The formula itself expresses a sum of additively twisted Fourier coefficients of automorphic forms on $\GL_{n}$ in terms of a sum over corresponding dual terms. The beauty of this transformation is that the dual side exhibits cancellation in the original summation by redistributing the weight of the summation. This structural shift is intrinsically tied to the functional equation of $\GL_{n}\times\GL_{1}$-$L$-functions.

The Vorono\u{\i} summation formula is an essential tool in the analytic theory of automorphic forms. It features critically in the study of moments of $L$-functions and subconvexity results \cite{assing-yet-another,li-subconvexity-annals} but also in more esoteric problems, such as reciprocity formulae for spectral averages \cite{blomer-khan-reciprocity-1,blomer-kahn-reciprocity-2,blomer-li-miller} and bounding norms of automorphic forms \cite{blomer-four,buttcane-khan}. Vorono\u{\i} himself \cite{voronoi-1} constructed the first such prototype formulae with the motivation of giving beyond trivial bounds for the error terms in Dirichlet's divisor problem and Gau{\ss}' circle problem.\footnote{We recommend the historical account of Miller--Schmid \cite{miller-schmid-book} who describe how the original results of Vorono\u{\i} are related to the summation formulae of the present day.}

Despite the ubiquity of the Vorono\u{\i} summation formula, full generality in the level aspect has previously not been addressed: this is the goal of the present work. Our results describe such general summation formulae from two perspectives:

\begin{itemize}

\item[(i)] A precise Vorono\u{\i} summation formula for $\GL_{n}$ in the most general setting, in particular allowing the level and modulus to ramify jointly. See Theorem \ref{thm:general-classical-intro} for sums of Hecke eigenvalues and Theorem \ref{thm:voronoi-general} for Whittaker functions.

\item[(ii)] We give two refinements of Theorem \ref{thm:general-classical-intro}: the first describes a convenient choice of test vector at which we explicate the Bessel transform; the second determines a sum in an arithmetic progression. See Corollaries \ref{cor:supercuspidal-intro} \& \ref{cor:ap-intro}.

\end{itemize}

Miller--Schmid \cite{miller-schmid-general} prove the level $N=1$ case of our formula in Theorem \ref{thm:general-classical-intro} without non-archimedean test functions. They attribute the difficulty of the level aspect to the lack of understanding of Atkin--Lehner theory for $\GL_{n}$-automorphic forms when $n> 2$. One naturally understands new and old forms as vectors which are respectively fixed by a filtration of compact open subgroups.

Our solution to this problem is to understand the corresponding dual summands locally, via a $p$-adic transform derived from the local functional equation for $\GL_{n}\times \GL_{1}$ at primes $p$ dividing the level. Our formulae in (i) describe a \textit{general} choice of vector in such transforms.
We go on to give a prototype result for a \textit{particular} choice of vector in (ii). This allows one to study the Fourier coefficients away from the level in a more refined and aesthetic way.

\subsection{A general classical formula}\label{sec:general-formula-intro}

An automorphic form $f$ on $\SL_{n}(\R)$, or rather on the adele group $\GL_{n}(\A_{\Q})$, naturally generates an automorphic representation of $\GL_{n}(\A_{\Q})$, which we denote by $\pi_{f}=\otimes_{v}\pi_{v}$. If such a form is a so-called newform then the associated representation is irreducible.

In this work we consider Maa\ss\ cusp forms on $\SL_{n}(\R)$ of level $N\geq 1$; these are the eigenvectors in $ L^{2}(\Gamma_{1}(N)\bs \SL_{n}(\R))$ of the generalised Laplacian whose constant terms are zero. If truth be told, we study just the Fourier coefficients $A_{f}\colon \Z^{n-1}\rightarrow\C$, as given in Definition \ref{def:classical-fourier-coeffs}. The normalisation of these Fourier coefficients coincides with Goldfeld's \cite[p.\ 260, (9.1.2)]{goldfeld}. For example, at least when the entries of $A_f$ are prime to $N$, they are the Hecke eigenvalues of $f$ realised as a product of Schur polynomials in the Satake parameters of $\pi_{f}$; see Lemma \ref{lem:shintani-classical}. The Ramanujan conjecture predicts that $A_f(m_1,\ldots,m_{n-1})\ll (m_1\cdots m_{n-1})^{\varepsilon}$. As we consider only cusp forms, we have $A_f(m_1,\ldots,m_{n-1})=0$ for $m_1\cdots m_{n-1}=0$.


Let $\chi$ be a (Dirichlet) character modulo $N$. For $f\in L^{2}(\Gamma_{1}(N)\bs \SL_{n}(\R))$ we say that $f$ has character $\chi$ if and only if
\begin{equation*}
f(\gamma g)=\chi(\gamma_{n,n})f(g)
\end{equation*}
for all $\gamma=(\gamma_{i,j})\in\SL_{n}(\Z)$ such that $\gamma_{n,i}\equiv 0 \Mod{N}$ for $1\leq i \leq n-1$. Without loss of generality, we always assume this property of $f$.\footnote{To avoid confusion, note that traditionally authors speak of $f$ having character $\chi$ with respect to the group $\Gamma_{0}(N)$, which is precisely the inference in the present article.} Moreover, $\chi$ determines the parity of $f$: $f$ is even, or respectively odd, if and only if $\chi$ is.


\begin{theorem}\label{thm:general-classical-intro}

Let $N\geq 1$ be an integer and $\chi$ a Dirichlet character modulo $N$. For $n\geq 2$ let $f\in L^{2}(\Gamma_{1}(N)\bs \SL_{n}(\R))$ be a Maa\ss\ cusp form with character $\chi$ and assume that $f$ is Hecke eigenform at each prime $p\nmid N$. Let $M\geq 1$ be an integer and for each $p\mid M$ choose some $\phi_{p}\in \Ccinf(\Qpx)$; also pick $\phi_{\infty}\in \Ccinf(\Rx)$. Let $c:=(c_{2},\ldots, c_{n-1})\in\Z^{n-2}$ such that $(c_{2}\cdots c_{n-1},MN)=1$. For $a,\ell,q\in\Z$ with $a\neq 0$, $\ell,q\geq 1$, $(a,\ell q)=(q,NM)=1$ and $\ell\mid (NM)^{\infty}$
we have the following Vorono\u{\i} summation formula:
\begin{equation*}\label{eq:main-voronoi-classical}
\begin{array}{r}\vspace{0.15in}
\displaystyle
\sum_{m\in\Z_{\neq 0}}
e\bigg(\frac{am}{\ell q}\bigg)
\frac{A_{f}(m,c_{2},\ldots,c_{n-1})}{\abs{m}^{\frac{n-1}{2}}}
\phi_{\infty}(m)\prod_{p\mid M}\phi_{p}(m)={q^{n-2}\prod_{i=2}^{n-1}\abs{c_{i}}^{(n-i)(\frac{i}{2}-1)}}\,\times\hspace{0.3in}
\\\vspace{0.15in}
\displaystyle
\sum_{\substack{m,r\in\Z_{\neq 0}\\(m,MN)=1\\r\mid (MN)^{\infty}}}
\sum_{d_{n-1}\mid qc_{2}}\, \sum_{d_{n-2}\mid \frac{q c_{2}c_3 }{d_{n-1}}}\cdots \sum_{d_2\mid \frac{q c_{2} \cdots c_{n-1}}{d_{n-1}\cdots d_{3}}}\KL\left(\overline{a\lambda}_{\ell}\ell r,m;q,c,d\right)\chi\left(\bar{m}\frac{qc_{2}\cdots c_{n-1}}{d_{n-1}\cdots d_{2}}\right)^{-1}
\\
\displaystyle\times\,
\frac{A_{f}(d_{n-1},\ldots,d_{2},m)}{\abs{m}^{\frac{n-1}{2}}\prod_{i=2}^{n-1}d_{i}^{\frac{i(n-i)}{2}}}\Bc_{\pi_\infty,\phi_{\infty}}\left(\frac{r m}{\lambda_{\ell}q^{n}} \prod_{i=2}^{n-1}\frac{ d_{i}^{i}}{c_{i}^{n-i}}\right)\prod_{p\mid MN}\Bc_{\pi_p,\Phi_{p}^{a/\ell q}}\left(\frac{r m}{\lambda_{\ell}q^{n}} \prod_{i=2}^{n-1}\frac{d_{i}^{i}}{c_{i}^{n-i}}\right)
\end{array}
\end{equation*}
where $\lambda_{\ell}:=[\ell,N]\ell^{n-1}L^{n}$ with $L$ denoting the largest square-free integer dividing $MN$ and $[\ell,N]:=\lcm(\ell,N)$; we fix inverses $a\bar{a}\equiv\lambda_{\ell}\bar{\lambda}_{\ell}\equiv 1\Mod{q}$ and $m\bar{m}\equiv 1 \Mod{N}$; the $(n-1)$-dimensional Kloosterman sum is defined by
\begin{equation}\label{eq:kloosterman-classical-def}
\begin{array}{l}\vspace{0.15in}
\displaystyle
\KL(x,y;q,c,d)\,=
\sum_{j=2}^{n-1}\sum_{\alpha_{j}\in \left(\Z/\left(\frac{qc_2 \cdots c_{n-j+1}}{d_{n-1}\cdots d_{i}}\right)\Z\right)^{\times}}
e\left((-1)^{n}\frac{x d_{n-1}\alpha_{n-1}}{q}\right)\hspace{0.7in}\\
\hfill\displaystyle\times\,e\left(\frac{d_{n-2}\alpha_{n-2}\bar{\alpha}_{n-1}}{\frac{qc_2} {d_{n-1}}}\right)\cdots e\left(\frac{d_{2}\alpha_{2}\bar{\alpha}_{3}}{\frac{qc_2\cdots c_{n-2}} {d_{n-1}\cdots d_{3}}}\right)
e\left( \frac{ y \bar{\alpha}_{2}}{\frac{qc_{2}\cdots c_{n-1} } {d_{n-1}\cdots d_{2}}}\right)
\end{array}
\end{equation}
for $x,y\in\Z$ and $d:=(d_{2},\ldots,d_{n-1})$; and the functions $\Bc_{\pi_\infty,\phi_{\infty}}$ and $\Bc_{\pi_p,\Phi_{p}^{a/\ell q}}$ are given explicitly in \eqref{eq:mellon-of-bessel-real} and \eqref{eq:mellon-of-bessel-p-adic}, respectively, whilst $\Phi_{p}^{a/\ell q}$ is defined in \eqref{eq:big-phi}.
\end{theorem}

The proof of Theorem \ref{thm:general-classical-intro} is detailed in \S \ref{eq:proof-of-main}. It is a specialisation of our more general summation formula in Theorem \ref{thm:voronoi-general}.

\begin{rem}[The $p$-adic Bessel transforms]

The functions $\Bc_{\pi_p,\Phi_{p}^{a/\ell q}}$ on $\Qpx$ are crucial as they allow the explicit evaluation of the dual terms in the Vorono\u{\i} formula at primes $p\mid N$. In this work we explicate these transforms. For particular choices of $f$, we show they are proportional to sums of twists of $\GL_{n}$-root numbers when the the level $N$ and modulus $\ell$ are jointly ramified $(\ell,N)>1$; see Corollary \ref{cor:supercuspidal-intro}.

\begin{itemize}

\item The term $\lambda_{\ell}=[\ell,N]\ell^{n-1}L^{n}$ is determined by a generic support condition Proposition, \ref{prop:bessel-support}, for $\Bc_{\pi_p,\Phi_{p}^{a/\ell q}}$. In special cases, this term and the $r\mid (MN)^{\infty}$-sum may be greatly improved; see Corollary \ref{cor:supercuspidal-intro}.

\item Upper bounds are sensitive to the test functions $\phi_{p}$ and the local representation $\pi_p$ attached to $\pi_p$. We shall see that the size of the functions $\Bc_{\pi_p,\Phi_{p}^{a/\ell q}}$ is related to sums of twisted $\GL_{n}$-epsilon constants.
\end{itemize}

\end{rem}

\begin{rem}[The real Bessel transforms]
If one chooses the support of $\phi_{\infty}$ to be contained in $\R_{>0}$ then Kowalski--Ricotta \cite[Prop.\ 3.5]{kowalski-ricotta} show that for any $A>0$ we have $\Bc_{\pi_\infty,\phi_{\infty}}(y)\ll y^{-A}$. (See Proposition \ref{prop:kowalski-ricotta} for further details.)
\end{rem}


\subsection{Refined summation formulae}\label{sec:refined-formula-intro}

Here we consider a certain family of Maa\ss\ forms $f$ by placing assumptions upon the local representations of $\pi_{f}=\otimes_{v}\pi_{v}$ at primes $v=p$ for which $p\mid N$.

\begin{ass}\label{ass:super-maass}
For each $p\mid N$, suppose that the twists of the local Euler factor $L(s,\chi\pi_{p})=1$, identically, for each character $\chi\colon\Qpx\rightarrow\Cx$. Moreover suppose that $\pi_p$ is twist minimal in the sense of Definition \ref{def:minimal}.
\end{ass}
The first part of this assumption is satisfied, for example, by all supercuspidal representations of $\GL_{n}(\Qp)$. A representation may be made twist minimal after twisting by a character. For example, a supercuspidal representation whose (log) conductor is indivisible by $n$ is always twist minimal by \cite[Prop.\ 2.2]{corbett-conductor}.

\subsubsection{An explicit factor at primes dividing the level}

\begin{cor}\label{cor:supercuspidal-intro}
Suppose, in addition to the hypotheses of Theorem \ref{thm:general-classical-intro}, that Assumption \ref{ass:super-maass} holds. Suppose too, for stylistic reasons\footnote{The formula is similar for general $\ell\mid N^{\infty}$; the $p$-adic Bessel transforms are more cumbersome for $v_{p}(\ell)\leq 1$ but are nevertheless given explicitly in Proposition \ref{prop:bessel-ox}.}, that $\ell$ is divisible by the squares of each of its prime divisors. Define the function
$$S_{f}(m;\frac{a}{\ell q})=\underset{\chi\Mod{\ell}}{{\sum}^{*}}\chi(-1)^{n-1}\chi(m\bar{a}q)\varepsilon(1/2,f\times \chi)\varepsilon(1/2,\chi^{-1})$$
where the notation $\sum^{*}$ indicates summation over just the primitive Dirichlet characters $\Mod{\ell}$ and $a\bar{a}\equiv 1 \Mod{\ell}$. Then we have the following refined Vorono\u{\i} summation formula:
\begin{equation*}
\begin{array}{r}\vspace{0.15in}
\displaystyle
\sum_{\substack{m\in\Z_{\neq 0}\\(m,N)=1}}
e\bigg(\frac{am}{\ell q}\bigg)
\frac{A_{f}(m,c_{2},\ldots,c_{n-1})}{\abs{m}^{\frac{n-1}{2}}}
\phi_{\infty}(m)\,=\, \frac{[N,\ell^n]^{\frac{n-2}{2}}\ell^{-1/2}q^{n-2}}{\prod_{p\mid N}(1-p^{-1})} S_{f}(m;\frac{a}{\ell q})
\hspace{0.3in}
\\\vspace{0.15in}
\displaystyle
\times\,\prod_{i=2}^{n-1}\abs{c_{i}}^{(n-i)(\frac{i}{2}-1)}\,
\sum_{\substack{m\in\Z_{\neq 0}\\(m,N)=1}}
\sum_{d_{n-1}\mid qc_{2}}\, \sum_{d_{n-2}\mid \frac{q c_{2}c_3 }{d_{n-1}}}\cdots \sum_{d_2\mid \frac{q c_{2} \cdots c_{n-1}}{d_{n-1}\cdots d_{3}}}\KL\left(\overline{a[N,\ell^{n}]}\ell ,m;q,c,d\right)\\
\displaystyle\times\,
\chi\left(\bar{m}\frac{qc_{2}\cdots c_{n-1}}{d_{n-1}\cdots d_{2}}\right)^{-1}\frac{A_{f}(d_{n-1},\ldots,d_{2},m)}{\abs{m}^{\frac{n-1}{2}}\prod_{i=2}^{n-1}d_{i}^{\frac{i(n-i)}{2}}}\Bc_{\pi_\infty,\phi_{\infty}}\left(\frac{m}{[N,\ell^{n}]q^{n}} \prod_{i=2}^{n-1}\frac{ d_{i}^{i}}{c_{i}^{n-i}}\right).
\end{array}
\end{equation*}
\end{cor}

\begin{proof}
This formula follows from the explicit computation of the Bessel transforms $\Bc_{\pi_p,\Phi_{p}^{a/\ell q}}$ under the refined hypotheses. This is executed in \S \ref{sec:bessel-special-cases}, specifically in Proposition \ref{prop:bessel-ox}.
\end{proof}

\begin{rem}
The occurence of $S_{f}(m;\frac{a}{\ell q})$ highlights the key feature of our method: this term should be thought of as a ``ramified'' Kloosterman sum. We arrived at this expression via the local functional equation and a Bessel transform. However, should one desire to unpick further, the sum of epsilon factors may be shown to equal a Kloosterman sum dependant on the inducing data of the supercuspidal representation $\pi_p$. The explication of these terms is a deep problem in the corresponding $p$-adic representation theory. At present, it seems only possible to proceed in a case-by-case fashion. A uniform treatment of such terms would be an impressive result.
\end{rem}

\subsubsection{Voronoi summation in arithmetic progressions}

An amusing variant of the main formula allows us to restrict the summation to an arithmetic progression. The formula we present here is well suited to applications such as those considered in \cite{kowalski-ricotta}. The advantage is given by the $p$-adic properties of the local Bessel functions.

\begin{cor}\label{cor:ap-intro}

With the hypotheses of Theorem \ref{thm:general-classical-intro} and Assumption \ref{ass:super-maass}, fix an integer $M\geq 1$ and for each $p\mid M$ and define $\phi_{p}=\Char_{1+M\Z_{p}}\in\Ccinf(\Qpx)$. Once again assume that $\ell$ is divisible by the squares of each of its prime divisors. Then
\begin{equation*}
\begin{array}{l}\vspace{0.15in}
\displaystyle
\sum_{\substack{m\in\Z_{\neq 0}\\(m,N)=1\\m\equiv 1\Mod{M}}}
e\bigg(\frac{am}{\ell q}\bigg)
\frac{A_{f}(m,c_{2},\ldots,c_{n-1})}{\abs{m}^{\frac{n-1}{2}}}
\phi_{\infty}(m)={q^{n-2}\prod_{i=2}^{n-1}\abs{c_{i}}^{(n-i)(\frac{i}{2}-1)}}\sum_{\substack{m\in\Z_{\neq 0}\\(m,MN)=1}}
\\\vspace{0.15in}
\displaystyle
\times\,\sum_{r\mid [M,\ell]}
\sum_{d_{n-1}\mid qc_{2}}\, \sum_{d_{n-2}\mid \frac{q c_{2}c_3 }{d_{n-1}}}\cdots \sum_{d_2\mid \frac{q c_{2} \cdots c_{n-1}}{d_{n-1}\cdots d_{3}}}\KL\left(\overline{ar}\ell ,m;q,c,d\right)\,\chi\left(\bar{m}\frac{qc_{2}\cdots c_{n-1}}{d_{n-1}\cdots d_{2}}\right)^{-1}
\\
\displaystyle\times\,
\frac{A_{f}(d_{n-1},\ldots,d_{2},m)}{\abs{m}^{\frac{n-1}{2}}\prod_{i=2}^{n-1}d_{i}^{\frac{i(n-i)}{2}}}\Bc_{\pi_\infty,\phi_{\infty}}\left(\frac{ m}{rq^{n}} \prod_{i=2}^{n-1}\frac{ d_{i}^{i}}{c_{i}^{n-i}}\right)\prod_{p\mid MN}\Bc_{\pi_p,\Phi_{p}^{a/\ell q}}\left(\frac{m}{rq^{n}} \prod_{i=2}^{n-1}\frac{d_{i}^{i}}{c_{i}^{n-i}}\right)
\end{array}
\end{equation*}
where the transforms $\Bc_{\pi_p,\Phi_{p}^{a/\ell q}}$ are described in Proposition \ref{prop:bessel-ap}, taking $k=v_p(M)$.
\end{cor}

\section{Background representation theory and notation}\label{sec:background}

Here we give a bite-sized recap of the theory of automorphic representations of $\GL(n)$. The main purpose of this section is to fix notation for \S \ref{sec:general}, and may be referred back to should confusion present itself. Its content may be sidestepped without causing injury to our discourse on Vorono\u{\i} summation.

\subsection{Notable matrix groups}\label{sec:matrix-groups}

For a commutative ring $R$ (with unit $1\in R$), let us introduce notation for certain subgroups and elements of $\GL_{n}(R)$. Let $B_{n}(R)=T_{n}(R)\ltimes U_{n}(R)$ denote the standard Borel subgroup, consisting of upper triangular matrices in $\GL_{n}(R)$, given by the semi-direct product of the maximal torus of diagonal matrices $T_{n}(R)$ and the subgroup $U_{n}(R)$ of upper triangular matrices whose $n$ eigenvalues are all equal to $1$. Denote the centre of $\GL_{n}(R)$ by $Z(R)\isom R^{\times}$, acting on $\GL_{n}(R)$ by scalar multiplication.

Let $1_{n}$ denote the $n\times n$ identity matrix. Let $w=w_{n}$ be the longest Weyl element of $\GL_{n}(R)$, defined recursively by $w_{n}=\smat{}{w_{n-1}}{1}{}$ and $w_{1}=1$. Define a second Weyl element by $w'=\smat{1}{}{}{w_{n-1}}$. We assign specialist notation to the matrices
\begin{equation}\label{eq:matrices-def}
a(y):=\begin{pmatrix}
y&\\
&1_{n-1}
\end{pmatrix}\in T_{n}(R)\andy
n(x):=\begin{pmatrix}
1&x&\\
&1&\\
&&1_{n-2}
\end{pmatrix}\in U_{n}(R).
\end{equation}
For $g\in \GL_{n}(R)$ we consider the involution $g^{\iota}:=\tran{g}{-1}$. We have $n(x)^{\iota}=\tran{n(x)}{-1}$ and $a(y)^{\iota}=a(y^{-1})$.
For a function $f\colon \GL_{n}(R)\rightarrow\C$, denote the action of the right regular action of $g\in \GL_{n}(R)$ by $\rho(g)f(x)=f(xg)$ for $x\in \GL_{n}(R)$.

\subsection{Local and global fields}\label{sec:local-global-fields}

Let $F$ be a number field. Let $\A_{F}$ denote the ring of $F$-adeles and $\of_{F}$ ring of algebraic integers contained in $F$. At each place $v$ of $F$ let $F_{v}$ denote the completion of $F$ at $v$. Let $S_{\infty}$ denote the set of archimedean places of $F$. Suppose $v\not\in S_{\infty}$. Then denote by $\of_{v}$ the ring of integers of $F_{v}$; $\pf_{v}$ the maximal (prime) ideal of $\of_{v}$; $\varpi_{v}$ a choice of uniformising parameter, that is a generator of $\pf_{v}$; and $q_{v}=\# (\of_{v}/\pf_{v})$. Let $\abs{\,\cdot\,}_{v}$ denote the absolute value on $F_{v}$, normalised so that $\abs{\varpi_{v}}_{v}=q_{v}^{-1}$. The place $v$ itself denotes the discrete valuation on $F_{v}$, satisfying $\abs{x}_{v} = q_{v}^{-v(x)}$ for $x\in F$. At real places $v\in S_{\infty}$ set $\abs{x}_{v}=\sgn(x)x$ for $x\in \R$ and at complex places $v\in S_{\infty}$ set $\abs{z}_{v}=z\bar{z}$ for $z\in\C$.

\subsection{Additive characters}\label{sec:additive-characters}

We say that an (additive) character $\psi=\otimes_{v}\psi_{v}$ of $\A_{F}/F$ is \textit{unramified} if $\of_{v}\subset \ker_{F_{v}}(\psi_{v})$ for each $v\not\in S_{\infty}$. We henceforth assume such a choice of $\psi$.
Note that the dual group of $\A_{F}/F$ is identified by the set $\{x\mapsto\psi(ax) : a\in F\}$. We shall abuse notation by extending $\psi$ to a character of each subgroup $H\leq U_{n}(\A_{F})$ as follows: if $h\in H$ is given by $h=(h_{i,j})$ then define
\begin{equation}\label{eq:add-unipotent-char}
\psi(h)=\psi(h_{1,2}+h_{2,3}+\cdots+h_{n-1,n}).
\end{equation}
By the $F$-invariance of $\psi$ we also have $(H\cap U_{n}(F))\subset \ker_{H} (\psi)$.


\subsection{The Whittaker model}\label{sec:whittaker-model}

let $\pi=\otimes_{v}\pi_{v}$ be an irreducible, cuspidal automorphic representation of $\GL_{n}(\A_{F})$, realised in a space of automorphic forms $\V_{\pi}$. There is a $\GL_{n}(\A_{F})$-intertwining map from $\V_{\pi}$ to the space of functions $W\colon \GL_{n}(\A_{F})\rightarrow \C$ that satisfy $W(ug)=\psi(u)W(g)$ for each $u\in U_{n}(\A_{F})$ carrying the right regular representation $\rho$ of $\GL_{n}(\A_{F})$. Explicitly, one maps $\varphi\in\V_{\pi}$ to the element
\begin{equation}\label{eq:whittaker-isomorphism}
W_{\varphi}\colon g\longmapsto \int_{U_{n}(F)\bs U_{n}(\A)}\varphi(ug)\overline{\psi(u)} du,
\end{equation}
where we choose $du$ to be the invariant probability measure on $U_{n}(F)\bs U_{n}(\A)$. Let $\Wh(\pi,\psi)$ denote the image of $\V_{\pi}$ under \eqref{eq:whittaker-isomorphism}, another irreducible $\GL_{n}(\A_{F})$-module. One calls $\Wh(\pi,\psi)$ the $\psi$-\textit{Whittaker model} of $\pi$. This generalises the classical realisation of Fourier coefficients.

Similarly, at each place $v$ of $F$ there is a notion of a $\psi_{v}$-Whittaker model; this is the space $\Wh(\pi_{v},\psi_{v})$ of functions $W\colon \GL_{n}(F_{v})\rightarrow \C$ satisfying $W_{v}(ug)=\psi_{v}(u)W_{v}(g)$ for each $u\in U_{n}(F_v)$. It is again a $\GL_{n}(F_{v})$-module under $\rho$ and one has a non-zero module homomorphism $\pi_{v}\isom \Wh(\pi_{v},\psi_{v})$. The uniqueness of Whittaker models
\footnote{Global multiplicity one for $\GL_{n}$ was proved by Shalika \cite[Theorem 5.5]{shalika-mult-one}. The local multiplicity one theorem was also completed by Shalika \cite[Theorem 3.1]{shalika-mult-one}, who showed that the non-archimedean result of Gelfand--Ka\v{z}dan \cite{gelfand-kazdan} was true for archimedean places too.}
implies that $\Wh(\pi,\psi)$ is given by the restricted tensor product $\Wh(\pi,\psi)=\otimes_{v}'\Wh(\pi_{v},\psi_{v})$. This equality is significant as it identifies $\Wh(\pi,\psi)$ as a space of \textit{factorisable} functions on $\GL_{n}(\A_{F})$.

It shall be of use to speculate on the support of non-archimedean Whittaker functions. The following lemma gives a first observation. In the unramified setting, we can say much more; see \S \ref{sec:shintani}.

\begin{lem}\label{lem:support-whittaker}
Suppose $v\not\in S_{\infty}$ and let $W_{v}\in \Wh(\pi_{v},\psi_{v})$ such that $W_{v}$ is $U_{n}(\of_{v})$-fixed.\footnote{For example, this condition is satisfied by all new and old forms.} Then, $W_{v}(a(y))=0$ for all $y\in \Fx_{v}$ with $\abs{y}_{v}>1$.
\end{lem}
\begin{proof}
Let $y\in \Fx_{v}$. Then, for all $x\in\of_{v}$ we have
\begin{equation*}
W_{v}(a(y)) = W(a(y)n(x)) = W(n(xy)a(y)) = \psi_{v}(xy)W(a(y)).
\end{equation*}
If $\abs{y}_{v}>1$, we can always find $x\in \of_{v}$ such that $\psi_{v}(xy) \neq 1$.
\end{proof}





\subsection{The Kirilov model}

The classical results of Gelfand--Ka\v{z}dan \cite{gelfand-kazdan}, if $v$ is non-archimedean, and Jacquet--Shalika \cite[Prop.\ 3.8]{jacquet-shalika-ep-1} imply that it is enough to consider the functions $W_{v}\in\Wh(\pi_{v},\psi_{v})$ restricted to the `mirabolic subgroup' $P_{n}(F_{v})$ of $\GL_{n}(F_{v})$ -- the stabiliser of $(0,\ldots,0,1)$ on the right -- to carry an irreducible representation of $\GL_n(F_{v})$ isomorphic to $\pi_{v}$: this space of vectors is known as the \textit{Kirilov model}.

\begin{prop}\label{prop:kirilov-schwartz}
Let $v$ be a place of $F$. Then the space of functions on $P_{n}(F_{v})$ given by $ W_{v}\vert_{P_{n}(F_{v})}$ for sum $W_{v}\in\Wh(\pi_{v},\psi_{v})$ contains the entire space of compactly supported Bruhat--Schwartz functions $\Phi$ on $P_{n}(F_v)$ such that $\Phi\left(\smat{n}{}{}{1}p\right)=\psi_{v}(n)\Phi(p)$ for each $n\in U_{n-1}(F_{v})$ and $p\in P_{n}(F_{v})$.
\end{prop}

\begin{rem}
In the present article, it suffices to consider just those Bruhat--Schwartz functions on $P_{n}(F_{v})$ with support on the matrices $a(y)$ for $y\in \Fx_{v}$.
\end{rem}

\begin{rem}
If $v$ is non-archimedean and $\pi_{v}$ is supercuspidal, then the space of functions on $P_{n}(F_{v})$ given by $ W_{v}\vert_{P_{n}(F_{v})}$ is \textit{precisely} the space of locally constant, compactly supported functions $\Phi$ described in Proposition \ref{prop:kirilov-schwartz}. In general, the co-dimension of these spaces is at most $n$.
\end{rem}

\subsection{Spherical Whittaker functionals}\label{sec:shintani}

In \cite{shintani}, Shintani evaluated the (spherical) Whittaker functionals on $\GL_{n}(F_{v})$ in terms of certain polynomials of the inducing data. Let $v$ be a non-archimedean place of $v$ and suppose $\pi_{v}$ is unramified\footnote{This criterion includes all but finitely many places of $F$.}; that is,
\begin{equation}\label{eq:spherical-prince}
\pi_{v}=\Ind_{B_{n}(F_{v})}^{\GL_{n}(F_{v})}(\mu_{1}\otimes\cdots\otimes\mu_{n})
\end{equation}
for unramified characters $\mu_{i}\colon\Fx\rightarrow\Cx$ for $i=1,\ldots,n$. Such characters are determined by their value on $\varpi_{v}$ and, in turn, $\pi_{v}$ is completely determined by the $n$ complex numbers (or `Satake parameters') $\mu_{i}(\varpi_{v})$ for $i=1,\ldots, n$.

Let $W^{\circ}_{v}\in\Wh(\pi_{v},\psi_{v})$ denote the unique $\GL_{n}(\of_{v})$-fixed vector with $W^{\circ}_{v}(1)=1$. Then, by the Iwasawa decomposition, $W^{\circ}_{v}$ is completely determined by its values on $T_{n}(F_{v})/T_{n}(\of_{v})$. These values are given in terms of a \textit{Schur polynomial}\footnote{It is of no coincidence that Schur polynomials appear as characters of the (finite-dimensional) irreducible representations of $\GL_{n}(\C)$ (see \cite[Chapter 6]{fultonrep} for instance).}, to which end we define the function
\begin{equation}\label{eq:schur-poly}
s_{\lambda}(t_{1},\ldots,t_{n}):=\prod_{1\leq i<j\leq n}(t_{i}-t_{j})^{-1}\,\det\left(
\begin{matrix}
t_{1}^{\lambda_{1}+n-1}&t_{2}^{\lambda_{1}+n-1}&\cdots&t_{n}^{\lambda_{1}+n-1}\\
t_{1}^{\lambda_{2}+n-2}&t_{2}^{\lambda_{2}+n-2}&\cdots&t_{n}^{\lambda_{2}+n-2}\\
\vdots&\vdots&\ddots&\vdots\\
t_{1}^{\lambda_{n}}&t_{2}^{\lambda_{n}}&\cdots&t_{n}^{\lambda_{n}}\\
\end{matrix}
\right),
\end{equation}
evaluated on a toral element $t=\diag(t_{1},\ldots,t_{n})\in\GL_{n}(\C)$ and indexed by a partition $\lambda=(\lambda_{1},\ldots,\lambda_{n})\in\Z^{n}$ satisfying
\begin{equation}\label{eq:inequality-cond}
\lambda_{1}\geq \lambda_{2}\geq \cdots\geq \lambda_{n}.
\end{equation}

\begin{prop}[Shintani's formula \cite{shintani}]\label{prop:shintani-formula}
Let $\lambda=(\lambda_{1},\ldots,\lambda_{n})\in\Z^{n}$ and consider the element $\varpi_{v}^{\lambda}:=\diag(\varpi_{v}^{\lambda_{1}},\ldots,\varpi_{v}^{\lambda_{n}})\in T_{n}(F_{v})$. Then $W^{\circ}_{v}(\varpi_{v}^{\lambda})=0$ unless $\lambda$ satisfies \eqref{eq:inequality-cond}, in which case
\begin{equation*}
W^{\circ}_{v}(\varpi_{v}^{\lambda}) = q_{v}^{\sum_{i=1}^{n}\lambda_{i}\left(\frac{n+1}{2}-i\right)} s_{\lambda}(\mu_{1}(\varpi_{v}),\ldots,\mu_{n}(\varpi_{v})).
\end{equation*}
\end{prop}

\begin{rem}
These values are equal to a certain Hecke eigenvalue of $W_{v}^{\circ}$. We refer to \cite[Lecture 7]{cogdell-fields} for wider exposition and context of Shintani's result. Note that the power of $q_{v}$ in Shintani's formula is precisely (the square-root of) the modular character of $B_{n}(F_{v})$. It comes directly from our \textit{unitary} normalisation of the unramified principal series $\pi_{v}$.

\end{rem}

\subsection{The contragredient representation}\label{sec:contragredient}

The involution $\iota$ on $\GL_{n}(\A_{F})$, given by $g^{\iota}=\tran{g}{-1}$, determines an injection of $\pi$ into its contragredient representation $\tilde{\pi}$. Explicitly by defining the functions $\tilde{\varphi}=\varphi\circ\iota$ for $\varphi\in\V_{\pi}$. By \eqref{eq:whittaker-isomorphism}, we obtain the $\bar{\psi}$-Whittaker function $W_{\tilde{\varphi}}$ which satisfies $W_{\tilde{\varphi}}(g)=W_{\varphi}(wg^{\iota})$ and
\begin{equation}\label{eq:contragredient-right-shift}
W_{(\rho(h)\varphi)}(g^\iota)=\rho(h^\iota)W_{\tilde{\varphi}}(g)
\end{equation}
for $g,h\in \GL_{n}(\A)$. The Whittaker model $\Wh(\tilde{\pi},\bar{\psi})$ (resp.\ $\Wh(\tilde{\pi}_{v},\bar{\psi}_{v})$) is then  given by the space of functions $\tilde{W}$, defined by $\tilde{W}(g)=W(wg^{\iota})$, for each $W\in\Wh(\pi,\psi)$ (resp.\ $\Wh(\tilde{\pi}_{v},\bar{\psi}_{v})$).

\subsection{Euler factors and epsilon constants}\label{sec:euler-factors}

Let $v$ be a place of $F$. For any character $\chi$ of $\Fx_{v}$ we can define the twist $\chi\pi_{v}=(\chi\circ\det)\otimes\pi_{v}$. We follow Godement--Jacquet \cite{godement-jacquet} in defining the local Euler factors $L(s,\chi\pi_{v})$, epsilon constants $\varepsilon(s,\chi\pi_{v},\psi_{v})$, and gamma factors
\begin{equation}\label{eq:def:gamma-factors}
\gamma(s,\chi\pi_{v},\psi_{v})=\varepsilon(s,\chi\pi_{v},\psi_{v})\dfrac{L(1-s,\chi^{-1}\tilde{\pi}_{v})}{L(s,\chi\pi_{v})}.
\end{equation}
For $v\not\in S_{\infty}$ we have the formula 
\begin{equation}\label{eq:epsilon-constant-conductor}
\varepsilon(s,\pi_{v},\psi_{v})=\varepsilon(1/2,\pi_{v},\psi_{v})\,q_{v}^{a(\pi_{v})(\frac{1}{2}-s)}
\end{equation}
(see \cite[Theorem 3.3 (4) \& (3.3.5)]{godement-jacquet}), in which the \textit{conductor}\footnote{By \cite[\S 5, Th{\'e}or{\`e}me]{jacquet-ps-shalika-conductor}, the conductor $a(\pi_{v})$ is realised as the smallest non-negative integer for which there is a non-zero vector in $\Wh(\pi_{v},\psi_{v})$ which is fixed by the corresponding maximally chosen compact open subgroup of $\GL_{n}(F_{v})$. Such a vector is unique up to scalar multiplication.} $a(\pi_{v})$ of $\pi_{v}$ is implicitly defined. The \textit{root number} $\varepsilon(1/2,\pi_{v},\psi_{v})$ takes its value on the unit circle.
Of course, these factors can also be reinterpreted via the local Langlands correspondence for $\GL_{n}$ (see \cite{wedhorn,rohrlich} for such a description); this viewpoint is useful should one want to compute them explicitly for a given $\pi_{v}$.
One has $a(\pi_{v})=0$ if and only if $\pi_{v}$ is unramified. Thus one makes sense of the following global definition.
\begin{defn}\label{def:level}
We call a positive integer $N(\pi)$ the \textit{level} (or \textit{analytic conductor}) of an irreducible automorphic representation $\pi=\otimes_{v}\pi_{v}$ if $N(\pi)=\prod_{v\not\in S_{\infty}} q_{v}^{a(\pi_{v})}$.
\end{defn}
Then the global $L$-function $L(s,\pi)=\prod_{v}L(s,\pi_{v})$, initially defined for $\Re(s)$ sufficiently large, has analytic continuation to all $s\in\C$, is bounded in vertical strips, and satisfies the functional equation $$L(s,\pi)=\varepsilon(1/2,\pi)N(\pi)^{(1/2-s)}L(1-s,\tilde{\pi})$$ where $\varepsilon(1/2,\pi)=\prod_{v} \varepsilon(1/2,\pi_{v},\psi_{v})$ is independent of $\psi$ (see \cite[Theorem 13.8]{godement-jacquet}).

\section{\texorpdfstring{Vorono\u{\i}}{Voronoi} summation via the Whittaker model}\label{sec:general}

Ichino--Templier's recasting of Vorono\u{\i} summation formulae in an adelic framework is a tremendously important step in understanding the mechanisms governing such identities. Here we prove a strengthening of their results as given in \cite{ichino-templier}.


Throughout this section, fix a number field $F$ and an unramified character $\psi=\otimes_{v}\psi_{v}$ of $\A_{F}/F$ (see \S \ref{sec:additive-characters}). Also fix $n\geq 2$ and let $\pi=\otimes_{v}\pi_{v}$ denote an irreducible, cuspidal automorphic representation of $\GL_{n}(\A_{F})$. We now pose a Vorono\u{\i} summation formula for the Fourier coefficients in the Whittaker model $\Wh(\pi,\psi)$.

\subsection{The generalised Bessel tranformation}\label{sec:bessel-existance}

We now describe the key tool in handling local ramification in the Vorono\u{\i} summation formula: the generalised Bessel tranformation. Presently we give a criterion for the existence of such a transform. But a key feature of our work is the explication of these transforms, for which we give full details in \S \ref{sec:bessel-explicit} in the general case.

Let $v$ be a place of $F$. If $v\in S_{\infty}$ (resp.\ $v\not\in S_{\infty}$) let $\Cc^{\infty}(\Fx_{v})$ denote the space of smooth (resp.\ locally constant) functions on $\Fx_{v}$. In either case, let $\Ccinf(\Fx_{v})\subset \Cc^{\infty}(\Fx_{v})$ denote the subspace of functions whose support is compact.

\begin{prop}\label{prop:bessel-existence}

Let $\Phi\colon\Fx_{v}\rightarrow\C$ be a function defined by $\Phi(y)=W(\smat{y}{}{}{1_{n}})$ for some $W\in\Wh(\pi_{v},\psi_{v})$; by Proposition \ref{prop:kirilov-schwartz}, this includes the set $\Ccinf(\Fx_{v})$. Then there exists a unique function $\Bc_{\pi_{v},\Phi}\colon \Fx_{v}\rightarrow\C$, the Bessel transform of $\Phi$, such that for all $s\in\C$ with $\Re(s)$ sufficiently large and for all characters $\chi$ of $\Fx_{v}$ we have
\begin{equation}\label{eq:duality-equation}
\begin{array}{l}\vspace{0.1in}
\displaystyle\int_{\Fx_{v}}\Bc_{\pi_{v},\Phi}(y)\chi(y)^{-1}\abs{y}_{v}^{s-\frac{n-1}{2}}\dxy\,=\hfill\\
\hspace{1in}\displaystyle
\chi(-1)^{n-1}\gamma(1-s,\chi\pi_{v},\psi_{v})\int_{\Fx_{v}}\Phi(y)\chi(y)\abs{y}_{v}^{1-s-\frac{n-1}{2}}\dxy
\end{array}
\end{equation}
where $\dxy$ denotes a Haar measure on $\Fx_{v}$.
\end{prop}

\begin{proof}
This result follows directly from the local functional for $\GL_{n}\times\GL_{1}$ as proved by Jacquet--Shalika \cite{jacquet-ps-shalika-rs-1,jacquet-shalika-rs-2}; cf.\ \cite[Lemma 5.2]{ichino-templier}.
\end{proof}

\begin{rem}
For $\Phi\in \Ccinf(\Fx_{v})$, the transform $\Bc_{\pi_{v},\Phi}$ satisfies the following properties: sub-polynomial decay at infinity (for each $\delta> 0$ we have $\tilde{\Phi}(y)\ll \abs{y}_{v}^{-\delta}$ as $\abs{y}_{v}\rightarrow\infty$) and polynomial growth at zero (there exists a $\delta> 0$ such that $\tilde{\Phi}(y)\ll \abs{y}_{v}^{\delta}$ as $y\rightarrow 0$).
\end{rem}

 

\subsection{The hyper Kloosterman sum}\label{sec:hyper-kloosterman}


At a place $v\not\in S_{\infty}$, denote by $T_{v}^{1}\subset T_{n-2}(F_{v})$ the set of diagonal matrices of the form $t=\diag(t_{2},\ldots, t_{n-1})$ with $\abs{t_{i}}_{v}\geq 1$ for each $2\leq i \leq n-1.$ Then $\of_{v}$ acts (additively) on $t_{i}$ if $\abs{t_{i}}_{v}>1$, whence we define the quotients 
\begin{equation}\label{eq:lambda-xi}
\Lambda_{t_{i}}=\begin{cases}
t_{i}\ofx_{v}/\of_{v}&\ifs \abs{t_{i}}_{v}>1\\
\{1\}&\ifs \abs{t_{i}}_{v}=1.
\end{cases}
\end{equation}
Now fix a `modulus' $\zeta=(\zeta_{v})\in \A_{F}$ and a `shift' $\xi=(\xi_{v})\in T_{n}(\A_{F})$. We write $\xi_{v}=\diag(\xi_{1},\cdots,\xi_{n})$. Let $R$ be a set of places $v\not\in S_{\infty}$ such that $\pi_{v}$ is unramified and either $\abs{\zeta_{v}}_{v}>1$ or $\xi_{v}\not\in T_{n}(\of_{v})$. For $v\in R$ let $t\in T_{v}^{1}$. If $\abs{\zeta_{v}\xi_{1}^{-1}\xi_{2}}_{v}\geq 1$ then for $y\in \Fx_{v}$ define the $(n-1)$-dimensional hyper-Kloosterman sum by
\begin{equation}\label{eq:kloosterman-sum-v1}
\begin{array}{l}\vspace{0.1in}
\displaystyle\Kl_{v}(y,t;\zeta,\xi)=\abs{\xi_{2}\zeta_{v}}^{n-2}\abs{\xi_{3}\cdots\xi_{n}}^{-1}_{v}\psi_{v}(-\xi_{2}\xi_{3}^{-1})\\
\hspace{0.4in}\displaystyle\times\,\sum_{i=2}^{n-1} \sum_{x_{i}\in\Lambda_{t_{i}}}\psi_{v}( (-1)^{n}y\zeta_{v}^{-1}\xi_{2}^{-1}\xi_{n} x_{n-1}^{-1}\cdots x_{2}^{-1})\prod_{j=2}^{n-1}\psi_{v}(\xi_{n-j+1}\xi_{n-j+2}^{-1}x_{j}).
\end{array}
\end{equation}
If $\abs{\zeta_{v}\xi_{1}^{-1}\xi_{2}}_{v}\leq 1$ then define $\Kl_{v}(y,t;\zeta,\xi)\,=\,\Kl_{v}(y,t;\xi_{1}\xi_{2}^{-1},\xi)$ so that
\begin{equation}\label{eq:kloosterman-sum-v2}
\begin{array}{l}\vspace{0.1in}
\displaystyle\Kl_{v}(y,t;\zeta,\xi)=\abs{\xi_{1}}^{n-2}\abs{\xi_{3}\cdots\xi_{n}}^{-1}_{v}\psi_{v}(-\xi_{2}\xi_{3}^{-1})\\
\hspace{0.4in}\displaystyle\times\,\sum_{i=2}^{n-1} \sum_{x_{i}\in\Lambda_{t_{i}}}\psi_{v}( (-1)^{n}y\xi_{1}^{-1}\xi_{n} x_{n-1}^{-1}\cdots x_{2}^{-1})\prod_{j=2}^{n-1}\psi_{v}(\xi_{n-j+1}\xi_{n-j+2}^{-1}x_{j}).
\end{array}
\end{equation}
More generally, define the product $T_{R}^{1}=\prod_{v\in R}T_{v}^{1}$, which we view embedded in $T_{n-2}(\A_{F})$ by extending trivially at places $v\not\in R$. Moreover, for $y\in\bigcap_{v\in R} F_{v}$ and $t=(t_{v})\in T_{R}^{1}$ define
\begin{equation}\label{eq:kloosterman-sum-R}
\Kl_{R}(y,t;\zeta,\xi)=\begin{cases}
\prod_{v\in R}\,\Kl_{v}(y,t_{v};\zeta,\xi)&\ifs R\neq \emptyset\\
1 & \ifs R=\emptyset.
\end{cases}
\end{equation}
Finally, for $t=(t_{v})\in T_{R}^{1}$ with $t_{v}=\diag(t_{2},\ldots, t_{n-1})$, define the matrix
\begin{equation}\label{eq:delta-def}
\delta_{R}(t;\zeta,\xi)=(\delta_{v})\in T_{n}(\A_{F})
\end{equation}
whose local factors are $\delta_{v}=1$ for $v\not\in R$;
\begin{equation*}
\delta_{v}=\begin{pmatrix}
\zeta_{v}^{-1}(\det t)^{-1}\xi_{2}^{-1}&&&&\\
&t_{2}\xi_{n}^{-1}&&&\\
&&\ddots&&\\
&&&t_{n-1}\xi_{3}^{-1}&\\
&&&&\zeta_{v}\xi_{1}^{-1}
\end{pmatrix}
\end{equation*}
for $v\in R$ such that $\abs{\zeta_{v}\xi_{1}^{-1}\xi_{2}}>1$; and
\begin{equation*}
\delta_{v}=\begin{pmatrix}
\xi_{1}^{-1}(\det t)^{-1}&&&&\\
&t_{2}\xi_{n}^{-1}&&&\\
&&\ddots&&\\
&&&t_{n-1}\xi_{3}^{-1}&\\
&&&&\xi_{2}^{-1}
\end{pmatrix}
\end{equation*}
for $v\in R$ such that $\abs{\zeta_{v}\xi_{1}^{-1}\xi_{2}}\leq 1$.

\begin{rem}
In practice, we only consider sums of finitely many terms in $T^{1}_{R}$, as determined by the support of unramified Whittaker new-vectors on $W_{v}(a(y)\delta_{v})$ for $v\in R$ (See Proposition \ref{prop:shintani-formula} and Theorem \ref{thm:voronoi-general}). In particular we consider
\begin{equation}\label{eq:whittaker-supp-xi}
\abs{y \zeta^{-1} (\det t_{v})^{-1}\xi_{2}^{-1}}_{v}\leq \abs{t_{2}\xi_{n}^{-1}}_{v}\leq\cdots\leq \abs{t_{n-1}\xi_{3}^{-1}}_{v}\leq \abs{\zeta_{v}\xi_{1}^{-1}}_{v}
\end{equation}
for $\xi_{v}=\diag(\xi_{1},\ldots,\xi_{n})$ and $\abs{\zeta_{v}\xi_{1}^{-1}\xi_{2}}_{v}\geq 1$. This is an artefact of the more general notion of a \textit{Kloosterman integral} as studied by Stevens \cite[Def.\ 2.6]{stevens-kloostermann}, wherefrom \eqref{eq:kloosterman-sum-v1} and \eqref{eq:kloosterman-sum-v2} are derived in the proof of Theorem \ref{thm:voronoi-general}.
\end{rem}


\subsection{The general \texorpdfstring{Vorono\u{\i}}{Voronoi} summation formula}\label{sec:general-formula}

We now state the most general Vorono\u{\i} formula for $\GL_{n}(\A_{F})$, extending the results of Ichino--Templier. In particular, what follows in Theorem \ref{thm:voronoi-general} subsumes their main results, \cite[Theorems 1, 3, \& 4]{ichino-templier}, as well as generalising them to the case of joint level--modulus ramification. We further refine this formula by later explicating the generalised Bessel transforms; see \S \ref{sec:bessel-explicit}.

\begin{theorem}\label{thm:voronoi-general}

Let $n\geq 2$ and let $\pi=\otimes_{v}\pi_{v}$ be an irreducible cuspidal automorphic representation of $\GL_{n}(\A_{F})$. Let $Q$ denote the set of places $v\not\in S_{\infty}$ at which $\pi_{v}$ ramifies, $a(\pi_{v})> 0$. Let $\psi=\otimes_{v}\psi_{v}$ be a non-trivial, unramified\footnote{Recall that $\psi$ is unramified if and only if $\of_{v}\subset \ker_{F_{v}} (\psi_{v})$ for each $v\not\in S_{\infty}$. No generality is lost by making this assumption due to the choices of $\zeta\in\A_{F}$, $\xi\in T_{n}(\A_{F})$ and $W\in\Wh(\pi,\psi)$.} additive character of $\A_{F}/F$. Choose two finite sets $Q,S$ of places of $F$ such that $S_{\infty}\subset S$; $Q\cap S=\emptyset$; and $Q\cup S$ contains each place $v\not\in S_{\infty}$ at which $\pi_{v}$ ramifies, that is $a(\pi_{v})> 0$. Moreover, for each $v\in S$ pick some $\phi_{v}\in \Ccinf(\Fx_{v})$. Fix a `modulus' $\zeta=(\zeta_{v})\in \A_{F}$ and a `shift' $\xi=(\xi_{v})\in T_{n}(\A_{F})$ such that $\xi_{v}=1$ for $v\in Q\cup S$. 
Denote by $R$ the set of places $v\not\in Q \cup S$ such that either $\abs{\zeta_{v}}_{v}>1$ or $\xi_{v}\not\in T_{n}(\of_{v})$. Finally, let $W=\otimes_{v}W_{v}\in\Wh(\pi,\psi)$ such that $W_{v}$ is right-$\GL_{n}(\of_{v})$-invariant for almost all $v\not\in S_{\infty}$. Then
\begin{equation}\label{eq:formula-general-theorem}
\begin{array}{l}\vspace{0.15in}
\displaystyle
\sum_{\gamma\in\Fx}
\psi(\gamma\zeta)\,
W\left(\begin{pmatrix}
\gamma&\\&1_{n-1}
\end{pmatrix}
\xi\right)
\prod_{v\in S} \phi_{v}(\gamma)\,=\\
\hspace{0.1in}\displaystyle
\sum_{\gamma\in\Fx}
\sum_{t\in T_{R}^{1}}
\Kl_{R}(\gamma,t;\zeta,\xi)
\,\tilde{W}_{Q\cup S}\left(
\mat{\gamma}{}{}{1_{n-1}}
\delta_{R}(t;\zeta,\xi)
\right)
\prod_{v\in Q\cup S}\Bc_{\pi_{v},\Phi_{v}^{\zeta_{v}}}(\gamma)
\end{array}
\end{equation}
where the hyper-Kloosterman sum $\Kl_{R}(\gamma,t;\zeta,\xi)$ is defined in \eqref{eq:kloosterman-sum-R}; $\delta_{R}(t,\zeta,\xi)$ in \eqref{eq:delta-def}; we define $\tilde{W}_{Q\cup S}(g):=\prod_{v\not\in Q\cup S}\tilde{W}_{v}(g_{v})$ for $g=(g_{v})\in \GL_{n}(\A_{F})$ where the dual vector is given by $\tilde{W}_{v}(g)=W_{v}(w_{n} \tran{g}{-1})$, as in \S \ref{sec:contragredient}; and lastly, for $v\in Q\cup S$, one defines the function $\Phi_{v}$ on $y\in\Fx_{v}$ by
\begin{equation}\label{eq:big-phi}
\Phi_{v}(y):=\begin{cases}\vspace{0.1in}
\phi_{v}(y)W_{v}\left(\smat{y}{}{}{1_{n-1}}\right)& \text{if } v\in S\\
W_{v}\left(\smat{y}{}{}{1_{n-1}}\right)& \text{if } v\in Q,
\end{cases}
\end{equation}
and its twists by $\zeta_{v}$ via $\Phi_{v}^{\zeta_{v}}(y):=\psi_{v}(y\zeta_{v})\Phi_{v}(y)$.
\end{theorem}

\begin{rem}[On the hypotheses]
The details with which we formulate Theorem \ref{thm:voronoi-general} impose no real restrictions on the generality of the result. The additional generality in comparisson to \cite{ichino-templier} may be acutely summarised as follows:
\begin{itemize}
\item The generality in our choice of $W\in\Wh(\pi,\psi)$ (cf.\ \cite[Theorem 4]{ichino-templier}).
\item The assumption that if $v\not\in S_{\infty}$ and $\abs{\zeta_{v}}_{v}>1$ then $v\not\in Q\cup S$.
\end{itemize}
\end{rem}



\begin{proof}
Our proof of Theorem \ref{thm:voronoi-general} closely follows \cite[\S 2]{ichino-templier}, to which we shall refer for the sake of concision. Here we outline the core of this argument and describe in detail the modifications we make; in particular to \cite[\S 2.6 \& \S 2.7]{ichino-templier}.

One starts with the following fundamental identity:
\begin{equation}\label{eq:basic-id}
\sum_{\gamma\in\Fx} W_{\varphi}\left(\begin{pmatrix}
\gamma&\\
&1_{n-1}
\end{pmatrix}\right) =
\sum_{\gamma\in\Fx} \int_{\A^{n-2}}\widetilde{W}_{\varphi}\left(
\begin{pmatrix}
\gamma&&\\
x&1_{n-2}&\\
&&1
\end{pmatrix}
\begin{pmatrix}
1&\\
&w_{n-1}
\end{pmatrix}
\right)dx
\end{equation}
for any $\varphi\in\V_{\pi}$, letting $\V_{\pi}$ denote the space of automorphic forms carrying the representation $\pi$. See \S \ref{sec:whittaker-model} for discussion on Whittaker functions. The identity \eqref{eq:basic-id} follows from \cite[Prop.\ 1.1 \& Lem.\ 2.1]{ichino-templier}. It also features crucially in the construction of the global functional equation for $\GL_{n}\times\GL_{1}$. (Although, as noted by the authors, a proof does not appear in the literature until that in \cite[\S 4]{ichino-templier}.)

The subsequent goal is to evaluate \eqref{eq:basic-id} for an appropriate choice of $\varphi\in\V_{\pi}$. Typically, one synthesises the left-hand side as desired and picks up the pieces on the right.

To reconstruct the left-hand side of \eqref{eq:formula-general-theorem}, we first pick a vector $\varphi'\in\V_{\pi}$ and, without loss of generality, suppose $W_{\varphi'}=\otimes_{v}W_{v}'$. Firstly, note that the right translate $\rho(n(\zeta)\xi)\varphi'$ satisfies 
\begin{equation}
W_{\rho(n(\zeta)\xi)\varphi'}(a(\gamma))=W_{\varphi'}(a(\gamma)n(\zeta)\xi)=\psi(\gamma\zeta)W_{\varphi'}(a(\gamma)\xi).
\end{equation}
(See \eqref{eq:matrices-def} for definitions of the matrices $a(y)$ and $n(x)$.) Next, recall that the element $W=\otimes_{v}W_{v}\in\Wh(\pi,\psi)$ has been selected in the hypotheses of Theorem \ref{thm:voronoi-general}. We impose our choice upon each $W_{v}'$ as follows:
\begin{itemize}
\item If $v\in S$, consider the function $\Phi_{v}(y)=\phi_{v}(y)W_{v}(a(y))$. By Proposition \ref{prop:kirilov-schwartz}, there exists $W_{v}'\in \Wh(\pi_{v},\psi_{v})$ such that $W_{v}'(a(y))=\Phi_{v}(y)$ for all $y\in \Fx_{v}$.
\item If $v\not\in S$ then directly choose $W_{v}'=W_{v}$. In addition, for $v\in Q$ define the function $\Phi_{v}(y)=W_{v}(a(y))$ to homogenise notation.
\end{itemize}
We thus choose the test vector $\varphi:=\rho(n(\zeta)\xi)\varphi'$. By construction, the left-hand side of \eqref{eq:basic-id} is equal to the left-hand side of \eqref{eq:formula-general-theorem}.

It remains to compute the right-hand side of \eqref{eq:basic-id} after applying $\varphi=\rho(n(\zeta)\xi)\varphi'$. As is common practice in any trace formula, one observes that the geometric integrals factorise into local components: for $y\in\Fx_{v}$ define
\begin{equation*}
\Hf_{v}(y;\zeta_{v},\xi_{v})=\int_{F_{v}^{n-2}}\tilde{W}_{v}\left(
\begin{pmatrix}
y&&\\
x&1_{n-2}&\\
&&1
\end{pmatrix}
\begin{pmatrix}
1&\\
&w_{n-1}
\end{pmatrix}
\tran{n(-\zeta_{v})}{}
\xi_{v}^{-1}
\right)dx.
\end{equation*}
Then the right-hand side of \eqref{eq:basic-id} is equal to $\sum_{\gamma\in\Fx} \prod_{v} \Hf_{v}(\gamma;\zeta_{v},\xi_{v})$. We divide the argument into two genres, depending on whether places are in or out of $Q\cup S$.

Suppose $v\not\in Q\cup S$. We first show that $\Hf_{v}(y;\zeta_{v},\xi_{v})$ is equal to a certain hyper-Kloosterman integral, and then refine this integral in terms of Kloosterman sums.\footnote{This case is included in \cite[Theorem 3]{ichino-templier}. However, their result is stated in terms of hyper-Kloosterman integrals: the details of an explicit formula, such as ours, are left to the reader.} Let us always denote the factors of $\xi=(\xi_{v})$ by $$\xi_{v}=\diag(\xi_{1},\xi_{2},\ldots,\xi_{n}).$$
Note that if $\abs{\zeta_{v}\xi_{1}^{-1}\xi_{2}}_{v}\leq 1$ then $\Hf_{v}(y;\zeta_{v},\xi_{v})=\Hf_{v}(y;\xi_{1}\xi_{2}^{-1},\xi_{v})$, since $$\xi_{v}\tran{n(-\zeta_{v})}{}\xi_{v}^{-1}=\tran{n(-\zeta_{v}\xi_{1}^{-1}\xi_{2})}{}.$$ In such (unramified) cases, the integral may be computed directly; for comparison see \cite[\S 2.5]{ichino-templier}. However, in the spirit of austerity, we proceed by executing our computations with the assumption $\abs{\zeta_{v}\xi_{1}^{-1}\xi_{2}}_{v}\geq 1$. For example, for almost all places $v$ we have $\xi_{v}\in T(\of_{v})$ and $\zeta_{v}\in\of_{v}$, in which case $\Hf_{v}(y;\zeta_{v},\xi_{v})=\Hf_{v}(y;1,1)$.

For $x\in F_{v}^{n-2}$, rewrite
\begin{equation*}
\begin{pmatrix}
1&&\\
x&1_{n-2}&\\
&&1
\end{pmatrix}=\sigma\begin{pmatrix}
1_{n-2}&x&\\
&1&\\
&&1
\end{pmatrix}\sigma^{-1}\quad\text{where}\quad\sigma:= \begin{pmatrix}
&1&\\
1_{n-2}&&\\
&&1
\end{pmatrix}.
\end{equation*}
Changing variables from $x$ to $y^{-1}x$ we obtain
\begin{equation*}
\begin{array}{l}\vspace{0.15in}
\Hf_{v}(y;\zeta_{v},\xi_{v})\,=\\
\displaystyle\hspace{0.2in}\abs{y}_{v}^{n-2}\int_{F_{v}^{n-2}}\tilde{W}_{v}\left(\sigma
\begin{pmatrix}
1_{n-2}&x&\\
&1&\\
&&1
\end{pmatrix}
\sigma^{-1}
a(y)
\begin{pmatrix}
1&\\
&w_{n-1}
\end{pmatrix}
\xi_{v}^{-1}
\tran{n(-\zeta_{v}\xi_{1}^{-1}\xi_{2})}{}
\right)dx.
\end{array}
\end{equation*}
Consider the commutation relations
\begin{equation*}
\sigma^{-1}a(y)\sigma\sigma^{-1}\mat{1}{}{}{w_{n-1}}\xi^{-1}\mat{1}{}{}{w_{n-1}}\sigma = \begin{pmatrix}
\xi^{-1}_{n}&&&&&\\
&\ddots &&&&\\
&&\xi^{-1}_{3}&&&\\
&&&&y\xi^{-1}_{1}&\\
&&&&&\xi^{-1}_{2}
\end{pmatrix}
\end{equation*}
and, after applying the Bruhat decomposition to $\tran{n(-\zeta_{v}\xi_{1}^{-1}\xi_{2})}{}$,
\begin{equation*}
\begin{array}{r}\vspace{0.1in}
\sigma^{-1}\mat{1}{}{}{w_{n-1}}\tran{n(-\zeta_{v}\xi_{1}^{-1}\xi_{2})}{}\mat{1}{}{}{w_{n-1}}\sigma =\begin{pmatrix}
1_{n-2}&&\\
&1&-\zeta_{v}^{-1}\xi_{1}\xi_{2}^{-1}\\
&&1
\end{pmatrix}\hspace{0.4in}\\
\times\,
\begin{pmatrix}
1_{n-2}&&\\&-\zeta_{v}^{-1}\xi_{1}\xi_{2}^{-1}&\\
&&-\zeta_{v}\xi_{1}^{-1}\xi_{2}
\end{pmatrix}
\begin{pmatrix}
1_{n-2}&&\\
&&-1\\
&1&-\zeta_{v}^{-1}\xi_{1}\xi_{2}^{-1}
\end{pmatrix},
\end{array}
\end{equation*}
where the final factor above and $\sigma^{-1}\mat{1}{}{}{w_{n-1}}$ are both elements of $\GL_{n}(\of_{v})$. By the right-$\GL_{n}(\of_{v})$-invariance of $W_{v}$ we obtain
\begin{equation*}
\Hf_{v}(y;\zeta_{v},\xi_{v})=\abs{y}_{v}^{n-2}\int_{F_{v}^{n-2}}
\tilde{W}_{v}\left(\sigma A(x)
\right)dx.
\end{equation*}
where
\begin{equation*}
A(x):=
\begin{pmatrix}
1_{n-2}&x&-y\zeta_{v}^{-1}x\\
&1&-y\zeta_{v}^{-1}\\
&&1
\end{pmatrix}
\begin{pmatrix}
\xi^{-1}_{n}&&&&&\\
&\ddots &&&&\\
&&\xi^{-1}_{3}&&&\\
&&&&-y\zeta_{v}^{-1}\xi^{-1}_{2}&\\
&&&&&-\zeta_{v}\xi^{-1}_{1}
\end{pmatrix}.
\end{equation*}
Note that
\begin{equation*}
\sigma \begin{pmatrix}
1_{n-2}&x&-y\zeta_{v}^{-1}x\\
&1&-y\zeta_{v}^{-1}\\
&&1
\end{pmatrix}=\begin{pmatrix}
1&&-y\zeta_{v}^{-1}\\
&1_{n-2}&-y\zeta_{v}^{-1}x\\
&&1
\end{pmatrix}\sigma \begin{pmatrix}
1_{n-2}&x&\\
&1&\\
&&1
\end{pmatrix}.
\end{equation*}
\begin{itemize}
\item If $n=2$ then
\begin{equation*}
\Hf_{v}(y;\zeta_{v},\xi_{v})=\psi_{v}(y\zeta^{-1}_{v})\tilde{W}_{v}\left(\mat{-y\zeta_{v}^{-1}\xi_{2}^{-1}}{}{}{-\zeta_{v}\xi_{1}^{-1}}\right).
\end{equation*}
\item If $n\geq 3$ then, writing $x=\tran{(x_{1},\ldots,x_{n-2})}{}\in F_{v}^{n-2}$,
\begin{equation*}
\begin{array}{l}\vspace{0.1in}
\displaystyle\Hf_{v}(y;\zeta_{v},\xi_{v})=\abs{\xi_{2}\zeta_{v}}^{n-2}\abs{\xi_{3}\cdots\xi_{n}}^{-1}_{v}\\
\displaystyle\hspace{0.8in}\times\,\int_{F_{v}^{n-2}}
\psi_{v}(-\xi_{2}\xi_{3}^{-1}x_{n-2})\tilde{W}_{v}\left(
\tau
\begin{pmatrix}
1_{n-2}&x&\\
&1&\\
&&1
\end{pmatrix}
\right)dx,
\end{array}
\end{equation*}
where we define the matrix
\begin{equation}\label{eq:tau-def}
\tau=\begin{pmatrix}
&1&\\
1_{n-2}&&\\
&&1
\end{pmatrix}
\begin{pmatrix}
\xi_{n}^{-1}&&&\\
&\ddots&&\\
&&\xi_{3}^{-1}&\\
&&&-y\zeta_{v}^{-1}\xi_{2}^{-1}&\\
&&&&-\zeta_{v} \xi_{1}^{-1}
\end{pmatrix}.
\end{equation}
\end{itemize}

If $n\geq 3$ then $\Hf_{v}(y;\zeta_{v},\xi_{v})$ is a particular `hyper-Kloosterman integral' \cite[Def.\ 2.6]{stevens-kloostermann}. It remains to express it as a sum of $(n-1)$-dimensional hyper-Kloosterman sums. We remark that the following calculation and subsequent result is not contained in \cite{ichino-templier}. Nevertheless, to sidestep an amassment of notation, we refer to and amend their lemmata directly.

Now follow \cite[Def.\ 6.2]{ichino-templier} and define $\Kl^{\IT}(\psi_{t},\psi',\tau)$ as there (with the additional superscript!) for the following new variables: $\tau$, as in \eqref{eq:tau-def}; $\psi_{t}(u):=\psi_{v}(tut^{-1})$ for $u\in U_{n}(F_{v})$, $t\in T_{n}(F_{v})$; and $\psi'(a):=\psi_{v}(-\xi_{2}\xi_{3}^{-1}a)$ for $a\in F_{v}$. (See \cite[Def.\ 2.10]{stevens-kloostermann} for the original definition.) Then, \cite[Theorem 2.12]{stevens-kloostermann} implies
\begin{equation*}
\Hf_{v}(y;\zeta_{v},\xi_{v})=\abs{\xi_{2}\zeta_{v}}^{n-2}\abs{\xi_{3}\cdots\xi_{n}}^{-1}_{v} \sum_{t\in T_{n}(F_{v})/T_{n}(\of_{v})} \tilde{W}_{v}(t) \Kl^{\IT}(\psi_{t},\psi',t^{-1}\tau).
\end{equation*}
By \cite[Cor.\ 3.11]{stevens-kloostermann}, the Kloosterman sum $ \Kl^{\IT}(\psi_{t},\psi',t^{-1}\tau)$ factorises into an $(n-1)$-dimensional and a $1$-dimensional term by the decomposition of
\begin{equation*}
t^{-1}\tau=\begin{pmatrix}
&&&-y\zeta_{v}^{-1}\xi_{2}^{-1}t_{1}^{-1}&\\
\xi_{n}^{-1}t_{2}^{-1}&&&&\\
&\ddots&&&\\
&&\xi_{3}^{-1}t_{n-1}^{-1}&&\\
&&&&-\zeta_{v}\xi_{1}^{-1}t_{n}^{-1}\\
\end{pmatrix},
\end{equation*}
for $t=\diag(t_{1},\ldots,t_{n})\in T_{n}(F_{v})/T_{n}(\of_{v})$, into $\GL_{n-1}(F_{v})\times\GL_{1}(F_{v})$-parabolic factors. The $1$-dimensional factor is supported on $\abs{t_{n}}_{n}=\abs{\zeta_{v}\xi_{1}^{-1}}_{v}$, in which case it equals the constant $\psi_{v}(-\xi_{2}\xi_{3}^{-1})$.

We pick this moment to reorder summation by exchanging the variables $t_{i}$ with $t_{i}\xi_{n-i+2}$, for each $2\leq i \leq n-1$, and $t_{1}$ with $\xi_{2}t_{1}$. For the $(n-1)$-dimensional factor to be non-zero, by \cite[Th.\ 3.12]{stevens-kloostermann}, we require the determinant to be a unit, $\abs{t_{1}\cdots t_{n-1}}_{v}=\abs{y\zeta_{v}^{-1}}_{v}$, and each exposed sub-determinant to be integral. Checking the definition of an exposed sub-determinant \cite[Def.\ 3.3]{stevens-kloostermann}; one finds that this condition is equivalent to that $\abs{t_{i}}_{v}\geq 1$ for $2\leq i\leq n-1$. Collecting these observations we obtain the refined expression
\begin{equation*}
\begin{array}{l}\vspace{0.1in}
\displaystyle\Hf_{v}(y;\zeta_{v},\xi_{v})=\abs{\xi_{2}\zeta_{v}}^{n-2}\abs{\xi_{3}\cdots\xi_{n}}^{-1}_{v}\psi_{v}(-\xi_{2}\xi_{3}^{-1}) \\
\displaystyle\hspace{0.4in}\times\,\sum_{t\in T_{v}^{1}} \tilde{W}_{v}\left(\begin{pmatrix}
y\zeta_{v}^{-1}(\det t)^{-1}&&\\
&t&\\
&&\zeta_{v}
\end{pmatrix}
\begin{pmatrix}
&1&\\
&& w_{n-2}\\
1&&
\end{pmatrix}
\xi_{v}^{-1}\right) \Kl^{\IT}(\psi_{t'},\psi',\tau').
\end{array}
\end{equation*}
where $T_{v}^{1}$ was introduced in \S \ref{sec:hyper-kloosterman} and we define the variables
\begin{equation*}
t'=\mat{y\zeta_{v}^{-1}(\det t)^{-1}}{}{}{t}\begin{pmatrix}
\xi_{2}^{-1}&&&\\
&\xi_{n}^{-1}&&\\
&&\ddots&\\
&&&\xi_{3}^{-1}
\end{pmatrix};\quad 
\tau'=\mat{}{-\det t}{t^{-1}}{}.
\end{equation*}
The final step is to compute the terms $\Kl^{\IT}(\psi_{t'},\psi',\tau')$. This is executed recursively via \cite[Prop.\ 6.4]{ichino-templier}. Applying \cite[Cor.\ 6.5 \& Lem.\ 6.6]{ichino-templier} with the parameters $\tau'$ and $t'$ we obtain
\begin{equation*}
\Kl^{\IT}(\psi_{t'},\psi',\tau')=\sum_{x_{n-1}\in \Lambda_{t_{n-1}}} \psi_{v}(-\xi_{2}\xi_{3}^{-1}x_{n-1})\Kl^{\IT}(\psi_{t''},\psi'',\tau'')
\end{equation*}
where $\Lambda_{t_{i}}$ was introduced in \S \ref{sec:hyper-kloosterman}, $\psi''(a):=\psi_{v}(-\xi_{3}\xi_{4}^{-1}a)$ for $a\in F_{v}$, and we define the new variables
\begin{equation*}
\begin{array}{l}\vspace{0.1in}
t''=\begin{pmatrix}
y\zeta_{v}^{-1}(\det t)^{-1}\xi_{2}^{-1}&&&\\
&t_{2}\xi_{n}^{-1}&&\\
&&\ddots&\\
&&&t_{n-2}\xi_{4}^{-1}
\end{pmatrix};\\
\tau''=
\begin{pmatrix}
&&&x_{n-1}^{-1}(\det t)\\
t_{2}^{-1}&&&\\
&\ddots&&\\
&&t_{n-2}^{-1}&
\end{pmatrix}
.
\end{array}
\end{equation*}
Note that we have multiplied the top right-hand corner of $\tau'$ by $-x_{n-1}^{-1}$ to obtain the top right-hand corner of $\tau''$. Continuing recursively, terminating the evaluation with \cite[Lem.\ 6.6]{ichino-templier}, we deduce that
\begin{equation*}
\begin{array}{l}\vspace{0.1in}
\displaystyle \Kl^{\IT}(\psi_{t'\xi^{-1}},\psi',\tau')=\displaystyle\sum_{x_{n-1}\in \Lambda_{t_{n-1}}}\cdots\sum_{x_{2}\in \Lambda_{t_{2}}} \psi_{v}(-\xi_{2}\xi_{3}^{-1}x_{n-1})\prod_{j=2}^{n-2}\psi_{v}(\xi_{n-j+1}\xi_{n-j+2}^{-1}x_{j})\\
\hfill\times\,\psi_{v}\left((-1)^{n-1}y\zeta_{v}^{-1}\xi_{2}^{-1}\xi_{n}x_{n-1}^{-1}\cdots x_{2}^{-1}\right),
\end{array}
\end{equation*}
which is equal to $(\abs{\xi_{2}\zeta_{v}}^{n-2}\abs{\xi_{3}\cdots\xi_{n}}^{-1}_{v}\psi_{v}(-\xi_{2}\xi_{3}^{-1}) )^{-1}\Kl_{v}(y,t;\zeta,\xi)$, as defined in \S \ref{sec:hyper-kloosterman}, on the nose. (To be compared with \cite[Cor.\ 6.7]{ichino-templier}.) Note that, unlike \cite[Prop.\ 6.4]{ichino-templier}, we subsume the cases $\abs{t_{i}}_{v}>1$ and $\abs{t_{i}}_{v}=1$ into one via our definition of the sets $\Lambda_{t_{i}}$ in \eqref{eq:lambda-xi}. In particular, the above equality holds for $\abs{t_{2}}_{v}=1$ since then we have
\begin{equation*}
\abs{y\zeta_{v}^{-1}\xi_{2}^{-1}\xi_{n}x_{n-1}^{-1}\cdots x_{2}^{-1}}_{v}=\abs{y\zeta_{v}^{-1}\xi_{2}^{-1}\xi_{n}(\det t)^{-1}}_{v}\leq 1
\end{equation*}
from the support of $\tilde{W}_{v}$ (see \eqref{eq:whittaker-supp-xi}); indeed, $\psi_{v}$ is trivial on $\of_{v}$. As a last remark we simply note that for all places $v\not\in Q\cup R\cup S$ we have $$\Hf_{v}(y;\zeta_{v},\xi_{v})=\Hf_{v}(y;1,1)= \tilde{W}_{v}(a(y)).$$

Now consider the remaining places $v\in Q\cup S$. Theorem \ref{thm:voronoi-general} shall follow duly from the observation that 
$\Hf_{v}(y;\zeta_{v},\xi_{v})=\Bc_{\pi_{v},\Phi_{v}^{\zeta_{v}}}(y).$ 
Or argument follows \cite[\S 2.7]{ichino-templier}, except that we identify a different element of the Whittaker model upon shifting by the matrix $n(\zeta_{v})$.
By assumption, the function $\Phi_{v}^{\zeta_{v}}$ satisfies
\begin{equation*}
\Phi_{v}^{\zeta_{v}}(y)=\psi_{v}(y\zeta_{v})\Phi_{v}(y)=W_{v}'(n(\zeta_{v}y)a(y))=
W_{v}'(a(y)n(\zeta_{v})).
\end{equation*}
We denote this right-translate by $W_{v}^{\zeta_{v}}=\rho(n(\zeta_{v}))W_{v}'$, thus defining a new element $W^{\zeta}_{v}\in\Wh(\pi_{v},\psi_{v})$ satisfying $\Phi_{v}^{\zeta_{v}}(y)=W^{\zeta}_{v}(a(y))$ for all $y\in\Fx_{v}$. Now \cite[Lemma 2.3]{ichino-templier} applies to $\Phi_{v}^{\zeta_{v}}$, thus determining its (unique) transform so that 
\begin{equation*}
\Bc_{\pi_{v},\Phi_{v}^{\zeta_{v}}}(y)=\Hf_{v}(y;\zeta_{v},\xi_{v})
\end{equation*}
for all $y\in\Fx_{v}$. We remark that the hypotheses of \cite[Lemma 2.3]{ichino-templier} demand that $\Phi_{v}^{\zeta_{v}}\in\Ccinf(\Fx_{v})$; in the case $v\in S$ we have $\Phi_{v}^{\zeta_{v}}(y)=\psi_{v}(y\zeta_{v})\phi_{v}(y)W_{v}(a(y))$ is again smooth (resp.\ locally constant) of compact support. However, the argument there applies to all such functions $y\mapsto W_{v}(a(y))$ for any $W_{v}\in \Wh(\pi_{v},\psi_{v})$, in particular in the case $v\in Q$.



\end{proof}






\section{Explicit Bessel transforms}\label{sec:bessel-explicit}

Here we give an explicit description of the generalised Bessel transforms $\Bc_{\pi_{v},\Phi}$, as introduced in \S \ref{sec:bessel-existance}. We consider their analytic behaviour at all places and, for non-archimedean places $v$, we detail the `joint ramification' case with $\Phi=\Phi_{v}^{\zeta_{v}}$ for $\abs{\zeta_{v}}_{v}>1$. In this section let us retain the notation of \S \ref{sec:general}.

\subsection{The Mellon inversion formula for local fields}

Let $v$ be any place of $F$ and consider a Bruhat--Schwartz function $\Phi\colon \Fx_{v}\rightarrow\C$. The fundamental principle of harmonic analysis on locally compact abelian groups \cite{ramakrishnan-valenza} is to study the frequencies of $\Phi$ contained in the unitary dual group $\Fxh$; this is the set of continuous\footnote{Without further mention, any character on a locally compact group is assumed continuous.}, unitary characters on $\Fx_{v}$.

\begin{defn}\label{def:mellin-transform}
Let $\mu\in\Fxh_{v}$. The \textit{Mellin transform} of $\Phi$ is given by
\begin{equation*}
\Me(\Phi,\mu)=\int_{\Fx_{v}}\Phi(y)\mu(y)\dxy
\end{equation*}
for a Haar measure $\dxy$ on $\Fx_{v}$. Similarly, for a Schwartz function $\tilde{\Phi}\colon \Fxh_{v}\rightarrow \C$, the \textit{inverse Mellin transform} of $\tilde{\Phi}$ evaluated at $y\in \Fx_{v}$ is given by
\begin{equation*}
\Me^{-1}(\Phi,\mu)=\int_{\Fxh_{v}}\tilde{\Phi}(\mu)\mu(y)^{-1} d\mu
\end{equation*}
for a Haar measure $d\mu$ on $\Fxh_{v}$.
\end{defn}

\begin{prop}[Mellin inversion]\label{prop:mellin-inversion}
There exist `self-dual' normalisations of the pair of Haar measures $(\dxy,d\mu)$ such that
\begin{equation*}
\Me^{-1}\circ\Me=\Me\circ\Me^{-1}=\Id.
\end{equation*}
\end{prop}

\begin{proof}
See \cite[\S 3]{ramakrishnan-valenza} for instance.
\end{proof}

\subsubsection{Archimedean Mellin inversion}

Suppose that $v$ is real so that $F_{v}=\R$. Any unitary character on $\Rx$ is of the form $\mu=\sgn^{r}\abs{\,\cdot\,}_{v}^{it}\in\hat{\R}^{\times}$ for $r\in \{0,1\}$ and $t\in \R$. Let us normalise
\begin{equation*}
\Me(\Phi,\mu)=\int_{\Rx}\Phi(y)\sgn(y)^{r}\abs{y}_{v}^{it}\,\dxy
\end{equation*}
and
\begin{equation*}
\Me^{-1}(\tilde{\Phi},y)=\frac{1}{4\pi i}\sum_{r\in\{0,1\}}\int_{\R}\tilde{\Phi}(\sgn^{r}\abs{\,\cdot\,}_{v}^{it})\sgn(y)^{r}\abs{y}_{v}^{-it}\,dt
\end{equation*}
where $\dxy=\sgn(y)y^{-1}dy$ and $dy$, $dt$ both denote the Lebesgue measure on $\R$. With this choice of Haar measures, Proposition \ref{prop:mellin-inversion} holds.

\begin{rem}[The Mellin transform for $\R_{>0}$]
For $s\in\C$ and $\phi\in \Ccinf(\R_{>0})$ let $$\me(\phi,s):=\int_{0}^{\infty}\phi(y)y^{s-1}\,dy.$$ Defining $\Phi(y)=\phi(y)y^{\sigma}$ for $y>0$, with $\sigma=\Re(s)$ sufficiently large, and $\Phi(y)=0$ for $y\leq 0$ we have $\Me (\Phi,\sgn^{r}\abs{\,\cdot\,}_{v}^{\Im(s)})=\me(\phi,s)$, constant on $r\in \{0,1\}$. Proposition \eqref{prop:mellin-inversion} implies the usual Mellin inversion formula
\begin{equation}\label{eq:classical-mellon}
\phi(y)=\frac{1}{2\pi i}\int_{\Re(s)=\sigma}\me(\phi,s)\,y^{-s}\,ds.
\end{equation}
\end{rem}

Suppose that $v$ is complex so that $F_{v}=\C$. Expressing a complex number $z\in \Cx$ in polar coordinates, $z=\abs{z}_{v}^{1/2}e^{i\arg(z)}$, the unitary dual of $\Cx$ may be identified as $\hat{\C}^{\times}=\hat{\R}_{>0} \times \widehat{\R/\Z} \isom \R \times \Z.$ We omit the details of the complex case in this article.

\subsubsection{Non-archimedean Mellin inversion}

Let $v$ be non-archimedean. (Recall the notation for $F_{v}$ defined in \S \ref{sec:local-global-fields}.) We refer to Taibleson's book \cite[\S II.4]{taibleson} for any background material. Define
\begin{equation*}
\Xf_{v}=\lbrace\, \chi\colon\Fx_{v}\rightarrow\Cx \,\vert\, \chi(\varpi_{v})=1 \, \rbrace\subset \Fxh.
\end{equation*}
Then $\Xf_{v}$ is a discrete group, isomorphic to the unitary dual of $\ofx_{v}$. Considering the `polar coordinates' $y=u\varpi_{v}^{v(y)}$ for each $y\in \Fx_{v}$, one identifies
\begin{equation*}
\Fxh_{v}= \Xf_{v} \times \hat{\Z}\isom  \Xf_{v}\times \R/\Z.
\end{equation*}
Explicitly, any unitary character $\mu\in\Fxh_{v}$ is of the form $\mu=\chi\abs{\,\cdot\,}_{v}^{it}$ for some $\chi\in\Xf_{v}$ and $t\in \R$ satisfying $-\pi/\log q_{v}<t\leq \pi/\log q_{v}$. We normalise measures by
\begin{equation*}
\Me(\Phi,\mu) = \sum_{k\in\Z}q_{v}^{ikt}\int_{\ofx_{v}}\Phi(y\varpi_{v}^{-k})\chi(y)\, \dxy
\end{equation*}
and
\begin{equation*}
\Me^{-1}(\tilde{\Phi},y) = \frac{\log q_{v}}{2\pi}\sum_{\chi\in\Xf_{v}}\chi(y)^{-1}\int_{-\pi/\log q_{v}}^{\pi/\log q_{v}}\tilde{\Phi}(\chi\abs{\,\cdot\,}_{v}^{it})\abs{y}_{v}^{-it}\, dt.
\end{equation*}

\subsection{Explicating the Bessel transform via Mellin inversion}

We now give a general and explicit description of the Bessel transforms introduced in \S \ref{sec:bessel-existance}. This expression is obtained by applying the Mellin inversion formula to the identity \eqref{eq:duality-equation} between $\Phi$ and its dual $\Bc_{\pi_{v},\Phi}$. Let $s=\sigma+ it\in\C$ where, as in \S \ref{sec:bessel-existance}, we assume $\sigma$ is sufficiently large. The left-hand side of \eqref{eq:duality-equation} is precisely the Mellin transform of the function $\abs{\,\cdot\,}_{v}^{\sigma-\frac{n-1}{2}}\cdot\Bc_{\pi_{v},\Phi}$ evaluated at $\chi^{-1}\abs{\,\cdot\,}_{v}^{it}\in\Fxh_{v}$. Explicitly,
\begin{equation}\label{eq:mellon-of-bessel-gen}
\begin{array}{l}\vspace{0.1in}
\displaystyle\Me(\abs{\,\cdot\,}_{v}^{\sigma-\frac{n-1}{2}}\cdot\Bc_{\pi_{v},\Phi},\chi^{-1}\abs{\,\cdot\,}_{v}^{it})\,=\\
\displaystyle\hspace{1in}\chi(-1)^{n-1}\gamma(1-s,\chi\pi_{v},\psi_{v})\int_{\Fx_{v}}\Phi(y)\chi(y)\abs{y}_{v}^{1-s-\frac{n-1}{2}}\dxy.
\end{array}
\end{equation}
Bifurcating according to the place $v$, we proceed by using Proposition \ref{prop:mellin-inversion} to invert this expression: we substitute \eqref{eq:mellon-of-bessel-gen} into the identity
\begin{equation}\label{eq:inversion-identity}
\Bc_{\pi_{v},\Phi}(y)=\abs{y}_{v}^{\frac{n-1}{2}-\sigma}\Me^{-1}(\mu\mapsto\Me(\abs{\,\cdot\,}_{v}^{\sigma-\frac{n-1}{2}}\cdot\Bc_{\pi_{v},\Phi},\mu),y).
\end{equation}

\subsubsection{Archimedean Bessel transforms and estimates}\label{sec:arch-bessel}

Let $v$ be a real-archimedean place of $F$ so that $F_{v}=\R$. Solving \eqref{eq:inversion-identity} with \eqref{eq:mellon-of-bessel-gen}, for all $y\in\Rx$ we have
\begin{equation}\label{eq:mellon-of-bessel-real}
\begin{array}{l}\vspace{0.15in}
\displaystyle\Bc_{\pi_{v},\Phi}(y)=
\frac{1}{4\pi i}\sum_{r\in\{0,1\}}(-1)^{r(n-1)}\sgn(y)^{r}\int_{\Re(s)=\sigma}\gamma(1-s,\sgn^{r}\pi_{v},\psi_{v})
\abs{y}_{v}^{\frac{n-1}{2}-s}\hspace{0.3in}\\
\displaystyle\hfill\times\,\int_{\Rx}\Phi(x)\sgn(x)^{r}\abs{x}_{v}^{1-s-\frac{n-1}{2}}\dxx\, ds.
\end{array}
\end{equation}

To quote Kowalski--Ricotta \cite[\S 3]{kowalski-ricotta} on the analytic behaviour of $\Bc_{\pi_{v},\Phi}$, we assume that $\Phi$ is compactly supported in $\R_{>0}$. The Bessel transform may then be expressed in terms of the classical Mellin transform;
\begin{equation*}
\begin{array}{l}\vspace{0.15in}
\displaystyle\Bc_{\pi_{v},\Phi}(y)=
\frac{1}{2}\sum_{r\in\{0,1\}}(-1)^{r(n-1)}\sgn(y)^{r}\\
\displaystyle\hspace{1in}\times\,\int_{\Re(s)=\sigma}\me\left(\Phi,1-s-\tfrac{n-1}{2}\right)\gamma(1-s,\sgn^{r}\pi_{v},\psi_{v})
\abs{y}_{v}^{\frac{n-1}{2}-s}ds.
\end{array}
\end{equation*}

Decomposing the operator $\Phi\mapsto\Bc_{\pi_{v},\Phi}$ into its $r$-summands, we find that they are unitary with respect to the $L^{2}$-norm on $\R_{>0}$ computed with respect to the Lebesgue measure \cite[Prop.\ 3.3]{kowalski-ricotta}; this depends on the parity of $n$. In \cite[Cor.\ 3.6]{kowalski-ricotta}, estimates are given for the sum of $\Bc_{\pi_{v},\Phi}$ juxtaposed with the Fourier coefficients of an automorphic form in an interval. Moreover, we record the following asymptotic estimates of \cite[Prop.\ 3.5]{kowalski-ricotta}.\footnote{In the notation of \cite[\S 3]{kowalski-ricotta}, our transform $\Bc_{\pi_{v},\Phi}$ coincides with their function ``$\Bc_{\alpha_{\infty}(f)}[w]$'' by assigning $w=\Phi$.}

\begin{prop}[Kowalski--Ricotta \cite{kowalski-ricotta}]\label{prop:kowalski-ricotta}
Suppose that the support of $\Phi$ is a compact subset of $\R_{>0}$. Then if $0< y \leq 1$ we have $$\Bc_{\pi_{v},\Phi}(y) \ll y^{-\left(\frac{1}{2}+\frac{1}{n^{2}+1}\right)}.$$ And for all $ y,A \in \R_{>0}$ we have $$\Bc_{\pi_{v},\Phi}(y)\ll_{A,\pi_{v},\Phi}y^{-A}.$$
\end{prop}


\subsubsection{Non-archimedean Bessel transforms}
Let $v$ be a non-archimedean place of $F$. The simultaneous solution of \eqref{eq:mellon-of-bessel-gen} and \eqref{eq:inversion-identity} implies that for all $y\in\Fx_{v}$ we have
\begin{equation}\label{eq:mellon-of-bessel-p-adic}
\begin{array}{l}\vspace{0.2in}
\displaystyle \Bc_{\pi_{v},\Phi}(y)=
\dfrac{ \log q_{v}}{ 2\pi}\, \sum_{\chi\in\Xf_{v}}\chi(-1)^{n-1}\chi(y)\int_{\sigma-\pi/\log q_{v}}^{\sigma+\pi/\log q_{v}}\gamma(1-s,\chi\pi_{v},\psi_{v})\,\abs{y}_{v}^{\frac{n-1}{2}-s}\hspace{0.3in}\\
\displaystyle\hspace{1.5in}
\hfill\times\, \int_{\Fx_{v}}\Phi(x)\chi(x)\abs{x}_{v}^{1-s-\frac{n-1}{2}}\,\dxx\, ds.\\
\end{array}
\end{equation}
This is the most general description of the Bessel transform.

\subsection{Non-archimedean Bessel transforms in detail}\label{sec:bessel-special-cases}

We now consider the Bessel transforms $\Bc_{\pi_{v},\Phi_{v}^{\zeta_{v}}}$, at a non-archimedean place $v$. In the general case, we give a bound for the support. However, in practice, we shall not always require the full generality of Theorem \ref{thm:voronoi-general}. Making an assumption on the local factor $\pi_{v}$, we give a refined formula for the Bessel transform. We then show how to choose the test functions $\phi_{v}$ to determine a Vorono\u{\i} formulae on arithmetic progressions.

A natural factor occurring at places for which $\abs{\zeta_{v}}_{v}>1$ is the \textit{Gau\ss}\ sum:
\begin{equation}\label{eq:gauss-sum}
\Ga_{v}(a,\chi):=\int_{\ofx_{v}}\psi_{v}(ay)\,\chi(y)\,\dxy.
\end{equation}
for $a\in\Fx_{v}$ and $\chi\in\Xf_{v}$.

\begin{lem}\label{lem:gauss-sum}
Let $a\in\Fx_{v}$ and $\chi\in\Xf_{v}$. 
If $a(\chi)=0$ (or equivalently $\chi=1$) then
\begin{equation*}
\Ga_{v}(a,1)=\left\lbrace\begin{array}{ll}\vspace{0.05in}
\quad 1& \ifs \abs{a}_{v}\leq 1 \\\vspace{0.05in}
\frac{1}{1-q_{v}}& \ifs \abs{a}_{v}= q_{v}\\
\quad 0 & \ifs \abs{a}_{v}>q_{v} .
\end{array}\right.
\end{equation*}
If $a(\chi)>0$ then $\Ga_{v}(a,\chi)=0$ unless $\abs{a}_{v}=q_{v}^{-a(\chi)}$, in which case
\begin{equation*}
\Ga(a,\chi)=\zeta(1)\abs{a}^{-1/2}\chi(a)^{-1}\varepsilon(1/2,\chi^{-1}).
\end{equation*}
\end{lem}

\begin{proof}
This result is well-known. For instance, a proof is given by the author in \cite[Lemma 2.3]{corbett-saha} using \cite[(7.6) \& Lem.\ 7-4]{ramakrishnan-valenza}.
\end{proof}

\subsubsection{Support of the Bessel transforms}

Let $v\in Q\cup S$ and let $\psi_{v}$, $W_{v}$, $\Phi_{v}$, and $\Phi_{v}^{\zeta_{v}}$ be as in Theorem \ref{thm:voronoi-general}. There are no additional assumptions in the following, in particular not on $\pi_{v}$.

\begin{prop}\label{prop:bessel-support}
Let $y\in \Fx_{v}$. If $\abs{\zeta_{v}}_{v}\leq 1$ then $\Bc_{\pi_{v},\Phi_{v}^{\zeta_{v}}}(y)= 0 $ whenever $\abs{y}_{v}> q_{v}^{n+a(\pi_{v})}$. If $\abs{\zeta_{v}}_{v}> 1$ then $\Bc_{\pi_{v},\Phi_{v}^{\zeta_{v}}}(y)= 0 $ whenever $\abs{y}_{v}> \abs{\zeta_{v}}_{v}^{n-1} q_{v}^{n+\max\{a(\pi_{v}), -v(\zeta_{v})\}}.$
\end{prop}

\begin{proof}
This follows from a direct computation of \eqref{eq:mellon-of-bessel-p-adic}. We proceed by assuming $v\in Q$ and then amend our arguments to cover the additional notation in the case $v\in S$ (with an arbitrary test function $\phi_{v}$). By Lemma \ref{lem:support-whittaker} on the support of $y\mapsto W_{v}(a(y))$, the inner integral becomes
\begin{equation*}
\int_{\Fx_{v}}\Phi(x)\chi(x)\abs{x}_{v}^{1-s-\frac{n-1}{2}}\,\dxx=\sum_{r\geq 0}W_{v}(a(\varpi_{v}^{r}))\int_{\ofx_{v}}\psi_{v}(\zeta_{v}\varpi_{v}^{r}u)\chi(u)q^{-r(1-s-\frac{n-1}{2})} d^{\times}u.
\end{equation*}
By Lemma \ref{lem:gauss-sum}, we may apply the support of the Gau\ss\ sum $\Ga_{v}(\zeta_{v}\varpi_{v}^{r},\chi)$. To solve the $s$-integral we use the formulas \eqref{eq:def:gamma-factors} and \eqref{eq:epsilon-constant-conductor} to evaluate the terms $\gamma(1-s,\chi\pi_{v},\psi_{v})$. We also write down a generic geometric series for the quotient of $L$-factors. This has no lower powers of $q_{v}^{-s}$ than $(q_{v}^{-s})^{-n}$ (thus we incur this factor in the support bound). The estimate then follows from bounding the conductor of $a(\chi\pi_{v})$ using \cite[Theorem 2.7]{corbett-conductor}.
\end{proof}

\subsubsection{An explicit formula for minimal supercuspidal representations}

Let us now enforce the following.

\begin{ass}\label{ass:supercuspidal}
Let $\pi_{v}$ satisfy $L(s,\chi\pi_{v})=1$, identically, for all $\chi\in \Xf_{v}$ with $a(\chi)\leq\max\{-v(\zeta_{v}),0\}$.
\end{ass}
This assumption is satisfied, for example, by all supercuspidal representations of $\GL_{n}(F_{v})$.

\begin{defn}\label{def:minimal}
We say $\pi_{v}$ is \textit{twist minimal} if $a(\pi_{v})=\min\{a(\chi\pi_{v}):\chi\in\Xf_{v}\}$.
\end{defn}
Tautologically, the property of being twist minimal may always be obtained via twisting by some $\chi\in\Xf_{v}$. For example, any supercuspidal representation $\pi_{v}$ such that $n\nmid a(\pi_{v})$ is twist minimal by \cite[Prop.\ 2.2]{corbett-conductor}.  Without loss of generality, we further impose that, for $v\in S$, $W_{v}(a(y))=\Char_{\supp(\phi_{v})}(y)$ so that $\Phi_{v}=\phi_{v}$. This may be chosen by Proposition \ref{prop:kirilov-schwartz}.






\begin{prop}\label{prop:bessel-ox}
Assume that $\pi_{v}$ is twist minimal (Definition \ref{def:minimal}) and satisfies Assumption \ref{ass:supercuspidal}. Let $\phi_{v}=\Char_{\ofx}$. Then $\Bc_{\pi_{v},\phi_{v}^{\zeta_{v}}}(y)=0$ if $\abs{y}_{v}\neq q_{v}^{\max\{a(\pi_{v}),-nv(\zeta_{v})\}}.$ Otherwise, suppose $\abs{y}_{v}= q_{v}^{\max\{a(\pi_{v}),-nv(\zeta_{v})\}}$.
If $\abs{\zeta_{v}}_{v}\leq 1$ then
\begin{equation*}
\Bc_{\pi_{v},\phi_{v}^{\zeta_{v}}}(y)=\varepsilon(1/2,\pi_{v},\psi_{v})q_{v}^{a(\pi_{v})\frac{n-2}{2}}.
\end{equation*}
If $\abs{\zeta_{v}}_{v}= q_{v}$ then
\begin{equation*}
\begin{array}{l}\vspace{0.1in}
\displaystyle\Bc_{\pi_{v},\phi_{v}^{\zeta_{v}}}(y)=\varepsilon(1/2,\pi_{v},\psi_{v})\frac{q_{v}^{a(\pi_{v})\frac{n-2}{2}}}{1-q_{v}}\,+\\
\displaystyle\hspace{0.6in}\frac{q_{v}^{\max\{a(\pi_{v}),n\}\frac{n-2}{2}+\frac{1}{2}}}{1-q_{v}^{-1}}\sum_{\substack{\chi\in\Xf_{v}\\a(\chi)=1}}\chi(-1)^{n-1}\chi(\zeta_{v}^{-1}y)\varepsilon(1/2,\chi\pi_{v},\psi_{v})\varepsilon(1/2,\chi^{-1},\psi_{v}).
\end{array}
\end{equation*}
If $\abs{\zeta_{v}}_{v}>1$ then
\begin{equation*}
\begin{array}{l}\vspace{0.1in}
\displaystyle\Bc_{\pi_{v},\phi_{v}^{\zeta_{v}}}(y)=\frac{1}{1-q_{v}^{-1}}q_{v}^{\max\{a(\pi_{v}),-nv(\zeta_{v})\}\frac{n-2}{2}+\frac{v(\zeta_{v})}{2}}\\
\displaystyle\hspace{1.2in}\times\,\sum_{\substack{\chi\in\Xf_{v}\\a(\chi)=-v(\zeta_{v})}}\chi(-1)^{n-1}\chi(\zeta_{v}^{-1}y)\varepsilon(1/2,\chi\pi_{v},\psi_{v})\varepsilon(1/2,\chi^{-1},\psi_{v}).
\end{array}
\end{equation*}
\end{prop}

\begin{proof}
We compute the expression \eqref{eq:mellon-of-bessel-p-adic} under Assumption \ref{ass:supercuspidal} so that
\begin{equation*}
\Bc_{\pi_{v},\phi_{v}^{\zeta_{v}}}(y)=\sum_{\chi\in\Xf}\Ga_{v}(\zeta_{v},\chi)\varepsilon(1/2,\chi\pi_{v},\psi_{v})\chi(-1)^{n-1}\chi(y)q_{v}^{a(\chi\pi_{v})\frac{n-2}{2}}\delta(v(y),-a(\chi\pi_{v})).
\end{equation*}

We evaluate this expression by the explicit formula for the Gau\ss\ sum in Lemma \ref{lem:gauss-sum}. In particular, the assumption that $\pi_{v}$ is minimal allows us to use the formula
\begin{equation*}
a(\chi\pi_{v})=\max\{a(\pi_{v}),na(\chi)\}
\end{equation*}
given in \cite[Prop.\ 2.2]{corbett-conductor}.
\end{proof}

Estimating trivially we obtain the following upper bound.

\begin{lem}\label{lem:bessel-ox-triv}

For $y\in\Fx_{v}$ we have
\begin{equation*}
\Bc_{\pi_{v},\phi_{v}^{\zeta_{v}}}(y)\ll q_{v}^{\max\{a(\pi_{v}),-nv(\zeta_{v})\}\frac{n-2}{2}+\frac{3v(\zeta_{v})}{2}}.
\end{equation*}

\end{lem}




\subsubsection{Summation in arithmetic progressions}

Let us maintain Assumption \ref{ass:supercuspidal} for simplicity and fix the test function $$\phi_{v}=\Char_{1+\varpi_{v}^{k}\of_{v}}$$ for some $k\geq 1$. 
Consider the following variant of the Gau\ss\ sum:
\begin{equation}\label{eq:guss-variant}
\Ga_{v}^{k}(a,\chi):=\int_{1+\varpi_{v}^{k}\of_{v}}\psi_{v}(ay)\chi(y)\dxy.
\end{equation}
For now, to determine the support of $\Ga_{v}^{k}(a,\chi)$ we only consider its size.
\begin{lem}\label{lem:bessel-ap-chi-support}
Suppose $\abs{a}_v>1$. Then $\Ga_{v}^{k}(a,\chi)=0$ unless $a(\chi)\leq\max\{k,-v(a)\}$, in which case $$\abs{\Ga_{v}^{k}(a,\chi)}^{2}=\Vol(1+\varpi_{v}^{k}\of_v,\dxy)\Vol(1+\varpi_{v}^{\max\{k,-v(a)\}}\of_v,\dxy)\ll q_{v}^{\max\{2k,k-v(a)\}}.$$
\end{lem}

\begin{proof}
Expanding the integral, orthogonality of additive characters implies that $$\abs{\Ga_{v}^{k}(a,\chi)}^{2}=\Vol(1+\varpi_{v}^{k}\of_v,\dxy)\int_{(1+\varpi_{v}^{k}\of_v)\cap (1+\varpi_{v}^{-v(a)}\of_v)}\chi(y)\dxy.$$ 
Orthogonality of multiplicative characters now verifies the lemma.
\end{proof}

\begin{prop}\label{prop:bessel-ap}
Under Assumption \ref{ass:supercuspidal}, assume $\phi_{v}=\Char_{1+\varpi_{v}^{k}\of_{v}}$. Then
\begin{equation*}
\Bc_{\pi_{v},\phi_{v}^{\zeta_{v}}}(y)=\sum_{\chi\in\Xf}\Ga^{k}_{v}(\zeta_{v},\chi)\varepsilon(1/2,\chi\pi_{v},\psi_{v})\chi(-1)^{n-1}\chi(y)q_{v}^{a(\chi\pi_{v})\frac{n-2}{2}}\delta(v(y),-a(\chi\pi_{v})).
\end{equation*}

\end{prop}

\begin{proof}
Unfolding definitions as before, we recover the the expression after noting that $\Ga^{k}_{v}(\zeta_{v},\chi)$ is supported on $a(\chi)\leq\max\{k,-v(\zeta_{v})\}$ by Lemma \ref{lem:bessel-ap-chi-support}.
\end{proof}

\begin{cor}
For minimal representations, $\Bc_{\pi_{v},\phi_{v}^{\zeta_{v}}}(y)$ is supported on the compact set $v(y)=\max\{a(\pi_{v}),nr\}$ for $0\leq r\leq\max\{k,-v(\zeta_{v})\}$. 
\end{cor}







\section{A classical formulation}\label{sec:classical}

In this final section, we translate our results into a more classical parlance; that of Maa\ss\ forms on $\SL_{n}(\R)$. We apply our representation theoretic results to such forms by considering the special case $F=\Q$.



\subsection{Specialist notation}

\subsubsection{Valuations}

Recall that for $F=\Q$ there is a single archimedean place and it is denoted by $\infty$. Here we have $\Q_{\infty}=\R$ and the absolute value is the usual one: $\abs{y}_{\infty}=\abs{y}:=\sgn(y)y$. All other places are non-archimedean and indexed by a rational prime $p$, denoting this property by $p<\infty$. For integers $a,b\geq 1$, we make the convention that $a\mid b^{\infty}$ if and only if $a\mid b^{k}$ for some $k\geq 0$. For all $a,b\in\Z$ define $[a,b]:=\lcm(\abs{a},\abs{b})$ and $(a,b):=\gcd(\abs{a},\abs{b})$ as usual.

\subsubsection{The standard additive character}

We use the notation $e(z):=e^{2\pi i z}$ for $z\in\C$. Fix the character $\psi\colon \A_{\Q}/\Q\rightarrow\C$ given by $\psi=\otimes_{v}\psi_{v}$ with $\psi_{\infty}(x)=e(-x)$ for $x\in\R$ and, for $x_{p}\in\Q_{p}$, $\psi_{p}(x_{p})=1$ if and only if $x_{p}\in\Zp$. All characters of $\A_{\Q}/\Q$ are of the form $x\mapsto \psi(ax)$ for some $a\in \Q$.




\subsection{Level structure}

There is a natural right-action of the group $\SL_{n}(\Z)$ on $(\Z/N\Z)^{n}$, whence we introduce the congruence subgroups $$\Gamma_{1}(N):=\Stab_{\SL_{n}(\Z)}((0,\ldots,0,1))\subset\SL_{n}(\Z)$$ for each integer $N\geq 1$, thus defining a filtration of $\SL_{n}(\Z)$-subgroups with respect to successive multiples of $N$. We also introduce the following $p$-adic analogues: by the right-action of $\GL_{n}(\Zp)$ on $(\Z/N\Z)^{n}$ we define $$K_{1}(N)_{p}:=\Stab_{\GL_{n}(\Zp)}((0,\ldots,0,1))\subset\GL_{n}(\Zp)$$ for each integer $N\geq 1$. Also define $K_{1}(N)=\{1\}\times\prod_{p<\infty}K_{1}(N)_{p}\subset \GL_{n}(\A_{\Q})$. We thus realise $\Gamma_{1}(N)$ embedded into $\GL_{n}(\A_{\Q})$ by
\begin{equation}\label{eq:gamma-1-embedding}
\Gamma_{1}(N)= \GL_{n}(\Q) \cap (\GL_{n}(\R)^{+}\times K_{1}(N)).
\end{equation}
These filtrations are a good choice on which to study the level structure of $\SL_{n}(\R)$-Maa\ss\ forms since there is a robust theory of newforms.\footnote{For $n=2$ the newforms theory originated with Atkin--Lehner \cite{atkin-lehner} and was developed by Casselman \cite{casselman}. For $n\geq 2$ this theory has been constructed by Gelfand–-Ka\v{z}dan \cite{gelfand-kazdan}. Jacquet--Piatetski-Shapiro--Shalika \cite{jacquet-ps-shalika-conductor} prove that the conductor associated to a newform is that which occurs in the $\varepsilon$-factor of the local functional equation \cite[Theorem 3.3]{godement-jacquet}.}

\subsection{Dirichlet and Hecke characters}\label{sec:central-character}

Central characters of automorphic representations are given by Hecke chearacters of $\Ax_{\Q}$. Briefly recall here the correspondence between Dirichlet characters $\chi$ of $\Mod{N}$ and finite order Hecke characters of conductor at most $N$. (See \cite[\S 12.1]{knightly-li} for a clear reference.) Explicitly, we define a character $\omega\colon \Ax_{\Q}/\Qx\rightarrow\Cx$ using the strong approximation theorem $\Ax_{\Q}=\Qx\cdot(\R_{>0}\times \prod_{p<\infty}\Zpx)$ by
\begin{equation}\label{eq:dirichlet-hecke-def}
\omega(y)=\prod_{p\mid N} \chi (y_p)
\end{equation}
where $y=(y'_v)\in \R_{>0}\times \prod_{p<\infty}\Zpx$ and $y_p\in (\Z/p^{v_{p}(N)}\Z)^{\times}$ is the image of $y'_p$ for $p\mid N$ obtained via the isomorphism $\Zpx/(1+N\Zp)\isom (\Zp/N\Zp)^{\times}$. As for any continuous Hecke character we have the  factorisation $\omega=\otimes_{v}\omega_{v}$. In particular, for each $p\nmid N$ and $y\in\Qpx$ we have
\begin{equation}\label{eq:dirichlet-hecke-at-p}
\omega_{p}(y)=\chi(p)^{-v_{p}(y)}.
\end{equation}
Moreover, for each integer $d\geq 1$ with $(d,N)=1$ we have $\prod_{p\mid d}\omega_{p}(d)= \chi(d)^{-1}$.

\subsection{Lifting Maa\ss\ forms to adele groups}

The dictionary between classical Maa\ss\ forms and automorphic forms on $\GL_{n}(\A_{\Q})$ hinges on the following strong approximation theorem:
\begin{equation}\label{eq:strong-approx}
\GL_{n}(\A_{\Q})\isom\GL_{n}(\Q)\cdot(\GL_{n}(\R)\times K_{1}(N))
\end{equation}
for each $N\geq 1$ (cf. \cite[Prop.\ 13.3.3]{goldfeld-hundley-2}). Explicitly, given $f\in L^{2}(\Gamma_{1}(N)\bs \SL_{n}(\R))$, we define a function $\varphi_{f}\in L^{2}(\GL_{n}(\Q)\bs\GL_{n}(\A_{\Q})/K_{1}(N))$ by
\begin{equation}\label{eq:adelisation}
\varphi_{f}(\gamma g_{\infty}k)= f(g_{\infty})
\end{equation}
for $\gamma\in \GL_{n}(\Q)$, $g_{\infty}\in\GL_{n}(\R)$ and $k\in K_{1}(N)$. Note that this definition is well defined by \eqref{eq:gamma-1-embedding}. Then $\varphi_{f}$ generates the automorphic representation $\pi_{f}$ of $\GL_{n}(\A_{\Q})$ with the central character $\omega:=\pi_{f}\vert_{Z(\A)}$. As in Theorem \ref{thm:general-classical-intro}, without loss of generality we assume that $\omega$ corresponds to a Dirichlet character $\chi\Mod{N}$.

Moreover, we now have a notion of (normalised) $\psi_{p}$-Whittaker function $W_{\varphi_{f}}$ associated to $f$, as in \S \ref{sec:whittaker-model}. Under the assumption that $f$ is a Hecke eigenform with respect to the operators $T_{p}$ (as defined in \cite[(9.3.5)]{goldfeld} for instance) for each $p\nmid N$ we have that
\begin{equation}\label{eq:whittaker-classical}
W_{\varphi_{f}}=W_{\infty}\otimes\bigotimes_{p<\infty}W_{p}.
\end{equation}
We fix the ongoing assumption that $W_{p}(1)=1$ for all primes $p<\infty$ so that \eqref{eq:whittaker-classical} imposes a constraint on the normalisation of $W_{\infty}\in \Wh(\pi_{\infty},\psi_{\infty})$, the so-called `Jacquet Whittaker function'.\footnote{This assumption only concerns finitely many primes $p$ since, by definition, $W_{p}=W_{p}^{\circ}$ is the unique $\GL_{n}(\Zp)$-fixed vector satisfying $W_{p}(1)=1$ for almost all $p\nmid N$ (cf.\ Proposition \ref{prop:shintani-formula}).}


\subsection{Fourier coefficients and Whittaker functions}\label{sec:classical-coeffs}

Henceforth, let us consider a cuspidal Maa\ss\ form $f\in L^{2}(\Gamma_{1}(N)\bs \SL_{n}(\R))$ which is a Hecke eigenform with respect to the operators $T_{p}$ for each $p\nmid N$. Without loss of generality, suppose $W_{\varphi_{f}}=\otimes_{v}W_{v}$ such that $\prod_{p<\infty}W_{p}(1)=1$, as before.

\begin{defn}\label{def:classical-fourier-coeffs}
For $(m_{1},\ldots,m_{n-1})\in\Z^{n-1}$ with $\prod_{i}m_{i}\neq 0$ let
\begin{equation*}
A_{f}(m_{1},\ldots,m_{n-1})=\prod_{i=1}^{n} \abs{m_{i}}^{\frac{i(n-i)}{2}}\prod_{p<\infty}W_{p}\left(\begin{pmatrix}
m_{1}\cdots m_{n-1}&&&\\
&&&\\
&\ddots&&\\
&&m_{n-1}&\\
&&&1
\end{pmatrix}\right).
\end{equation*}
When $\prod_{i}m_{i}= 0$, we extend the definition by requiring $A_{f}(m_{1},\ldots,m_{n-1})=0$. 
\end{defn}

At least when $(m_{1}\cdots m_{n-1},N)=1$, the following lemma implies that the coefficients $A_{f}(m_{1},\ldots,m_{n-1})$ are the Hecke eigenvalues of $f$ (cf. \cite[Lecture 7]{cogdell}).

\begin{lem}[Shintani's formula]\label{lem:shintani-classical}
Consider a prime $p\nmid N$, so that local component $\pi_{p}$ of $\pi_{f}=\otimes_{v}\pi_{v}$ is an unramified principal series representation (as in \eqref{eq:spherical-prince}) Satake parameters $\mu_{1}(p),\ldots,\mu_{n}(p)$. Then for integers $k_{i}\geq 0$, $i=1,\ldots,n-1$ we have
$$A_{f}(p^{k_{1}},\ldots,p^{k_{n-1}})=s_{\lambda}(\mu_{1}(p),\ldots,\mu_{n}(p))$$
for the partition $\lambda=(\lambda_{1},\ldots,\lambda_{n-1}, 0)$ with $\lambda_{i}=k_i+\cdots +k_{n-1}$ for $1\leq i \leq n-1$.
\end{lem}

\begin{proof}
The (square root of the) modular character determines the constant $$ p^{\sum_{i=1}^{n}\lambda_{i}\left(\frac{n+1}{2}-i\right)}= \prod_{i=1}^{n} p^{k_{i}\frac{i(n-i)}{2}}.$$
Recalling the reciprocity of absolute values, $\abs{\gamma}_{\infty}=\prod_{p<\infty}\abs{\gamma}_p^{-1}$ for $\gamma\in\Qx$, the claim now follows immediately from Proposition \ref{prop:shintani-formula}.
\end{proof}

\begin{rem}[Dual Maa\ss\ forms]
Consider the isomorphism between $\pi_{f}$ and its contragredient given by mapping $\varphi$ to the function $\varphi^{\iota}(g):=\varphi(\tran{g}{-1})$. We define the \textit{dual Maa\ss\ form}\footnote{For $N=1$, dual Maa\ss\ forms are discussed in \cite[\S 9.2]{goldfeld}.} $f^{\iota}\in L^{2}(\Gamma_{1}(N)\bs \SL_{2}(\R))$. At least for $(m_{1}\cdots m_{n-1},N)=1$, we have
\begin{equation}\label{eq:dual-coefficients}
A_{f^{\iota}}(m_{1},\ldots,m_{n-1})=\chi(m_{1}\cdots m_{n-1})A_{f}(m_{n-1},\ldots,m_{1}).
\end{equation}
\end{rem}

Our definition of the terms $A_{f}(m_{1},\ldots,m_{n-1})$ coincides with that given by Goldfeld \cite[p.\ 260, (9.1.2)]{goldfeld} and Kowalski--Ricotta \cite[\S 2]{kowalski-ricotta}, whose coefficients are denoted by ``$A$'' and ``$a_{f}$'', respectively; there the assumption $N=1$ is enforced. In particular, by \cite[Lem.\ 9.1.3]{goldfeld} we have the trivial bound $$A_{f}(m_{1},\ldots,m_{n-1}) \ll \prod_{k=1}^{n-1}\abs{m_{i}}^{i(n-i)/2}.$$
We make the following additional observations.
\begin{itemize}

\vspace{0.1in}
\item The Fourier coefficients are multiplicative: $$\hspace{0.5in}A_{f}(m_{1}m_{1}',\ldots,m_{n-1}m_{n-1}')=A_{f}(m_{1},\ldots,m_{n-1})A_{f}(m_{1}',\ldots,m_{n-1}')$$ whenever $(m_{1}\cdots m_{n-1},\ m_{1}'\cdots m_{n-1}')=1$.

\vspace{0.1in}
\item The Fourier coefficients ``arithmetically'' normalised: $A_{f}(1,\ldots,1)=1$.

\vspace{0.1in}
\item The coefficients $A_{f}(m_{1},\ldots,m_{n-1})$ are precisely those occurring in the various $L$-series \cite{godement-jacquet} attached to $f$, or rather $\pi_{f}$. (See \cite[\S 2]{kowalski-ricotta} for a concise explanation of this remark.)

\end{itemize}

\subsection{Derivation of the classical summation formula}\label{eq:proof-of-main}

Here we give a proof of Theorem \ref{thm:general-classical-intro}, obtained by specialising our choices in Theorem \ref{thm:voronoi-general}.

\subsubsection{The landscape}

The integers at which the Vorono\u{\i} summation problem ramifies are denoted by $M,N\geq 1$ where $N$ is the level of $f$ and $M$ determines the set of primes\footnote{We could as well choose $M$ to be square-free as we are only concerned with the set of primes dividing it.} $p\mid M$ at which we choose local test functions $\phi_{p}\in\Ccinf(\Qpx)$, alongside $\phi_{\infty}\in \Ccinf(\Rx)$.

Define the modulus $\zeta=(\zeta_{v})\in\A_{\Q}$ by first choosing $a,\ell,q\in\Z$ with $a\neq 0$, $\ell,q\geq 1$, $(a,\ell q)=(q,NM)=1$ and $\ell\mid (NM)^{\infty}$. Then we take $\zeta_{\infty}=0$ and $\zeta_{p}=a/\ell q$ for each $p<\infty$. By our choice of additive character $\psi$, for each $\gamma\in\Qx$ we have
\begin{equation}\label{eq:psi-to-e}
\psi(\zeta\gamma)= e(a\gamma/\ell q).
\end{equation}
Following Theorem \ref{thm:voronoi-general}, define the set $$S=\left\lbrace\infty\right\rbrace\cup\left\lbrace\, p \text{ prime} \, :\, p\mid M\,\right\rbrace$$
and for $y\in\Qp$ recall the definitions $\Phi_{p}(y)=\phi_{p}W_{p}(a(y))$, at the primes $p\mid M$, and $\Phi_{p}(y)=W_{p}(a(y))$ for $p\mid N/(M,N)$. However, at $\infty$ we make the assumption that $W_{\infty}(a(y))=\Char_{\supp(\phi_{\infty})}(y)$ by Proposition \ref{prop:shintani-formula}, simply so that $\Phi_{\infty}=\phi_{\infty}$.
 
Define the shift $\xi=(\xi_{v})\in T_{n}(\A_{\Q})$ as follows: for $i=2,\ldots,n-1$ make a choice of integers $c_{i}\geq 0$ such that $(c_{i},MN)=1$. Then, without loss of generality, for $p\nmid MN$ let $\xi_{p}=\diag(\xi_{1},\ldots,\xi_{n})$ such that $\xi_{i}=c_{i}\cdots c_{n-1}$ for $2\leq i \leq n-1$, $\xi_{1}=\xi_{2}$, and $\xi_{n}=1$; for $v=\infty$ and $v=p$ with $p\mid MN$ take $\xi_{v}=1$.

\subsubsection{The left-hand side}

With our specialist assumptions, the left-hand side of \eqref{eq:formula-general-theorem} is equal to
\begin{equation}\label{eq:lhs-classical}
\sum_{\gamma\in\Qx}e\bigg(\frac{am}{\ell q}\bigg)\prod_{p<\infty}  W_{p}\left(\begin{pmatrix}
\gamma c_{2}\cdots c_{n-1}&&&&\\
&c_{2}\cdots c_{n-1}&&&\\
&&\ddots &&\\
&&&c_{n-1}&\\
&&&&1
\end{pmatrix}\right)\prod_{v\in S}\phi_{v}(\gamma).\\
\end{equation}
Here we see the choice of $\xi$ was to be contained in the support of $\prod_{p<\infty}W_{p}$. Moreover, the summands of \eqref{eq:lhs-classical} are non-zero only for those $\gamma\in\Qx$ such that $\abs{\gamma c_{2}\cdots c_{n-1}}_{p}\leq\abs{c_{2}\cdots c_{n-1}}_{p}$. Thus, letting $\gamma=m\in \Z\cap\Qx$ and applying Definition \ref{def:classical-fourier-coeffs} we obtain that \eqref{eq:lhs-classical} is equal to
\begin{equation}\label{eq:lhs-finished-classical}
\sum_{m\in\Z_{\neq 0}}e\bigg(\frac{am}{\ell q}\bigg)\frac{A_{f}(m,c_{2},\ldots,c_{n-1})}{\abs{m}^{\frac{n-1}{2}}\prod_{i=2}^{n-1}\abs{c_{i}}^{\frac{i(n-i)}{2}}}\phi_{\infty}(\gamma)\prod_{p\mid M}\phi_{p}(\gamma).
\end{equation}

\subsubsection{The right-hand side}

Firstly, reconsider the set $R$ in Theorem \ref{thm:voronoi-general}:
$$ R=\{\, p \text{ prime} \,:\, p\mid qc_{2}\cdots c_{n-1}\,\}.$$
Recalling that $\xi_{1}=\xi_{2}$ by assumption, the problem bifurcates according to whether $p\in R$ satisfies $p\mid q$ or not. On the right-hand side of \eqref{eq:formula-general-theorem}, one incurs the set $T_{R}^{1}=\otimes_{p\in R}T_{p}^{1}$; the Kloosterman sum $\Kl_{R}(\gamma,t;\zeta,\xi)$; the Whittaker functions $\tilde{W}_p$ at primes $p\nmid MN$; and the local Bessel transforms at $p\mid MN$. The support of the $T_{p}^{1}$-sum is determined by the Whittaker function, according to Proposition \ref{prop:shintani-formula}. In particular, the support contains only those $\diag(t_{2},\ldots,t_{n-1})\in T_{p}^{1}$ whence we make the following inductive change of variables
\begin{equation*}
t_{i}^{-1}=q\frac{c_{n-i+1}\cdots c_{2}}{d_{i}\cdots d_{n-1}}\text{ for some }d_i\mid q\frac{ c_{n-i+1} \cdots c_{2}}{d_{i+1}\cdots d_{n-1}}
\end{equation*}
with respect to $i=n-1,\ldots,2$. In this coordinate system, the Whittaker function on the right-hand side of \eqref{eq:formula-general-theorem} reads
\begin{equation}\label{eq:rhs-whittaker}
\begin{array}{l}\vspace{0.15in}
\displaystyle\prod_{p\nmid MN}\tilde{W}_{p}\left(\frac{1}{qc_{2}\cdots c_{n-1}}\begin{pmatrix}
\gamma\det t^{-1}q^2&&&&\\
&qc_{2}\cdots c_{n-1}t_{2}&&&\\
&&\ddots &&\\
&&&qc_2 t_{n-1}&\\
&&&&1
\end{pmatrix}
\right)\\
\displaystyle =\,\chi(qc_{2}\cdots c_{n-1})^{-1}\frac{A_{f^{\iota}}(m,d_{2},\ldots,d_{n-1})}{\abs{m}^{\frac{n-1}{2}}\prod_{i=2}^{n-1}d_{i}^{\frac{i(n-i)}{2}}}\\
\end{array}
\end{equation}
where $m\in\Z$ with $(m,MN)=1$ and $\gamma_{0}\in\Qx$ such that $\abs{\gamma_{0}}_{p}=1$ for each $p\nmid MN$, together determining the change of variables
$$\gamma=\gamma_{0}\frac{ m d_{2}\cdots d_{n-1}}{q^{n}\prod_{i=2}^{n-1}\left(\frac{c_{n-i+1}}{d_i}\right)^{i-1}}.$$
Noting our observation \eqref{eq:dual-coefficients} on the coefficients of the dual Maa\ss\ form $f^{\iota}$, \eqref{eq:rhs-whittaker} is equal to
\begin{equation}
\chi\left(\frac{qc_{2}\cdots c_{n-1}}{md_{n-1}\cdots d_{2}}\right)^{-1}\frac{A_{f}(d_{n-1},\ldots,d_{2},m)}{\abs{m}^{\frac{n-1}{2}}\prod_{i=2}^{n-1}d_{i}^{\frac{i(n-i)}{2}}}.
\end{equation}

To unwrap the hyper-Kloosterman sums (of \S \ref{sec:hyper-kloosterman}), first note that $\xi_{1}\xi_{2}^{-2}=1$. We have the following two cases $p\nmid q$, so that $\abs{\zeta_{p}}_{p}=\abs{a}_{p}\leq 1$, and $p\mid q$, so that $\abs{\zeta_{p}}_{p}=\abs{q}^{-1}_{p}> 1$. In either case, noting $\psi_p(-\xi_2\xi_3^{-1})=e(-c_2)=1$, we incur the following constant
\begin{equation*}
\prod_{p\mid q c_{2}\cdots c_{n-1}}\abs{\xi_2 q^{-1}}_p^{n-2}\abs{\xi_3\cdots\xi_n}_{p}^{-1}\psi_p(-\xi_2\xi_3^{-1})=\frac{q^{n-2}}{\prod_{i=2}^{n-1}\abs{c_i}^{n-i}}.
\end{equation*}
Applying the Chinese remainder theorem to the product $\Kl_R(\gamma,t;\zeta,\xi)$, let us make the following inductive change of variables in the $\Lambda_{t_{i}}$-sum: $x_{n-1}=x_{j}=\left(\frac{q c_2}{r_{n-1}}\right)^{-1}\alpha_{n-1}$ with $\alpha_{n-1}\in\left(\Z/\left( \frac{q c_2 }{r_{n-1}}
\right) \Z\right)^{\times}$  and $$x_{j}=\left(\frac{q c_2 \cdots c_{n-j+1}}{r_j\cdots r_{n-1}}\right)^{-1}\alpha_{j}\bar{\alpha}_{j+1}\text{  with  }\alpha_{j}\in \left(\Z/\left( \frac{q c_2 \cdots c_{n-j+1}}{r_j\cdots r_{n-1}}
\right) \Z\right)^{\times}$$ for each $n-2\geq j \geq 2$, noting the congruence $\alpha_{j+1}\bar{\alpha}_{j+1}\equiv 1 \Mod{ \frac{q c_2 \cdots c_{n-j+1}}{r_j\cdots r_{n-1}}}$ is well defined. The resulting sum $\Kl_R(\gamma,t;\zeta,\xi)$ contains the summands
\begin{equation*}
\prod_{p\mid q c_{2}\cdots c_{n-1}}\psi_{p}(\xi_{n-j+1}\xi_{n-j+2}^{-1}x_{j})=e\left(\frac{c_{n-j+1}\alpha_{j}\bar{\alpha}_{j+1}}{\frac{qc_{2}\cdots  c_{n-j+1}} {d_{n-1}\cdots d_{j}}}\right)=e\left(\frac{d_{j}\alpha_{j}\bar{\alpha}_{j+1}}{\frac{qc_2\cdots c_{n-j}} {d_{n-1}\cdots d_{j+1}}}\right),
\end{equation*}
for $n-2\geq j \geq 2$; the term $e\left(\frac{d_{n-1}\alpha_{n-1}}{q}\right)$; and
\begin{equation*}
\begin{array}{l}\vspace{0.1in}
\displaystyle\prod_{p\mid q c_{2}\cdots c_{n-1}}\psi_{p}( (-1)^{n}\gamma\zeta_{v}^{-1}\xi_{2}^{-1}\xi_{n} x_{n-1}^{-1}\cdots x_{2}^{-1})\\\vspace{0.15in}
\hspace{0.8in}\displaystyle=\, e\left((-1)^{n}\frac{\gamma_0 md_2\cdots d_{n-1}}{q^n\prod_{i=2}^{n-1}\left(\frac{c_{n-i+1}}{r_i} \right)^{i-1}}   \frac{\frac{\ell q}{a}}{c_{2}\cdots c_{n-1}  }q^{n-2}\prod_{i=2}^{n-1}\left(\frac{c_{n-i+1}}{r_i} \right)^{i-1}\bar{\alpha}_{2}\right)\\
\hspace{0.8in}\displaystyle=\,  e\left((-1)^{n} \frac{m\gamma_0  \frac{\ell}{ a} \bar{\alpha}_{2}}{\frac{qc_{2}\cdots c_{n-1} } {d_{n-1}\cdots d_{2}}}\right).
\end{array}
\end{equation*}
Shifting each $\alpha_j$-variable by $\ell\gamma_0 /a$, we remove it from the above exponential and recover it in the term $e\left(\frac{\bar{a}\ell d_{n-1}\alpha_{n-1}}{q}\right)$. Altogether we obtain that $$\Kl_R(\gamma,t;\zeta,\xi)=\KL_{n-1}(\bar{a}\ell\gamma_0, m;q,c,d)$$
where $a\bar{a}\equiv 1\Mod{q}$, $d:=(d_{2},\ldots,d_{n-1})$, $c:=(c_{2},\ldots,c_{n-1})$, and the classical Kloosterman sum was defined in \eqref{eq:kloosterman-classical-def}.

We now consider the $MN$-part of the summands $\gamma\in\Qx$, indexed by $\gamma_0$. Let $L$ denote the largest square free integer such that $L\mid MN$. Suppose that the product $\prod_{p\mid MN}\Bc_{\pi_p,\Phi_{p}^{a/\ell q}}(\gamma)$ is non-zero. Then Proposition \ref{prop:bessel-support} implies that there exists $r\mid (MN)^{\infty}$ such that $$\gamma_0=\frac{r}{[\ell,N]\ell^{n-1}L^{n}}.$$


Altogether, the right-hand side of \eqref{eq:formula-general-theorem} is equal to
\begin{equation*}
\begin{array}{l}\vspace{0.15in}
\displaystyle
\frac{q^{n-2}}{\prod_{i=2}^{n-1}\abs{c_i}^{n-i}}
\displaystyle\sum_{\substack{m\in\Z_{\neq 0}\\(m,MN)=1}}\sum_{r\mid (MN)^{\infty}}\sum_{d_{n-1}\mid qc_{2}}\, \sum_{d_{n-2}\mid \frac{q c_{2}c_3 }{d_{n-1}}}\cdots \sum_{d_2\mid \frac{q c_{2} \cdots c_{n-1}}{d_{n-1}\cdots d_{3}}}\KL(\overline{a\lambda}_{\ell}\ell,m;q,c,d)\hspace{0.2in}\\\vspace{0.1in}
\displaystyle\hfill\times\,\chi\left(\bar{m}\frac{qc_{2}\cdots c_{n-1}}{md_{n-1}\cdots d_{2}}\right)^{-1}\frac{A_{f}(d_{n-1},\ldots,d_{2},m)}{\abs{m}^{\frac{n-1}{2}}\prod_{i=2}^{n-1}d_{i}^{\frac{i(n-i)}{2}}}\Bc_{\pi_\infty,\phi_{\infty}}\left(\frac{r m d_{2}\cdots d_{n-1}}{\lambda_{\ell}q^{n}\prod_{i=2}^{n-1}\left(\frac{c_{n-i+1}}{d_i}\right)^{i-1}}\right)\\
\displaystyle\hfill\times\, \prod_{p\mid MN}\Bc_{\pi_p,\Phi_{p}^{a/\ell q}}\left(\frac{r m d_{2}\cdots d_{n-1}}{\lambda_{\ell}q^{n}\prod_{i=2}^{n-1}\left(\frac{c_{n-i+1}}{d_i}\right)^{i-1}}\right)
\end{array}
\end{equation*}
where $\lambda_{\ell}:=[\ell,N]\ell^{n-1}L^{n}$, $a\bar{a}\equiv\lambda_{\ell}\bar{\lambda}_{\ell}\equiv 1\Mod{q}$, and $m\bar{m}\equiv 1\Mod{N}$.
Thus we conclude the proof of Theorem \ref{thm:general-classical-intro}.


\bibliographystyle{amsplain}			
\bibliography{bibliography}				

\end{document}
